\documentclass[a4paper,11pt]{article}

\usepackage{fancyhdr}
\usepackage[numbers]{natbib}
\usepackage{float}
\usepackage{subfigure}
\usepackage{amsmath}
\usepackage{hyperref}
\usepackage{amsthm}
\usepackage{multirow}
\usepackage{booktabs}
\usepackage{bm}
\usepackage{amssymb}
\usepackage{algorithm}
\usepackage{algpseudocode}
\usepackage{ragged2e} 
\usepackage{empheq}
\usepackage{enumitem}
\usepackage{siunitx}

\usepackage[a4paper,left=2.cm,right=2.cm,top=3.cm,bottom=3.cm]{geometry}

\def\tsc#1{\csdef{#1}{\textsc{\lowercase{#1}}\xspace}}
\tsc{WGM}
\tsc{QE}

\newtheorem{theorem}{Theorem}
\newtheorem{lemma}[theorem]{Lemma}
\newtheorem{remark}{Remark}
\newtheorem{assumption}{Assumption}
\newtheorem{definition}{Definition}
\newtheorem{proposition}[theorem]{Proposition}

\begin{document}
\let\WriteBookmarks\relax
\def\floatpagepagefraction{1}
\def\textpagefraction{.001}



\title{A geometry-based deep equilibrium model for image restoration under multiplicative Gamma noise}  
\newcommand{\shorttitle}{A geometry-based deep equilibrium model}



%

\pagestyle{fancy}
\fancyhf{}
\fancyhead[L]{\shorttitle} 
\fancyhead[R]{\thepage}   
\renewcommand{\headrulewidth}{0.4pt}

\author{
    Shengkun Yang
    \thanks{School of Mathematics, Harbin Institute of Technology, Harbin, 150001, Heilongjiang, China. Email: 23B912006@stu.hit.edu.cn.}
    \and
    Luca Ratti
    \thanks{Department of Mathematics, University of Bologna, Bologna, 40126, Emilia-Romagna, Italy. 
    Email: luca.ratti5@unibo.it. Corresponding author.}
    \and
    Zhichang Guo
    \thanks{School of Mathematics, Harbin Institute of Technology, Harbin, 150001, Heilongjiang, China. Email: mathgzc@hit.edu.cn.}
}

\maketitle

\begin{abstract}
We propose a deep learning framework for image restoration from images degraded by both multiplicative Gamma noise and blur. Unlike conventional deep equilibrium (DEQ) models that rely on implicit neural regularization, the proposed method learns an explicit and interpretable regularizer parameterized by geometric priors associated with surface area and mean curvature. To minimize the resulting variational model, we develop a mirror descent algorithm tailored to the commonly used Gamma-noise fidelity terms. Leveraging the Kurdyka–\L{}ojasiewicz property for functions defined in $o$-minimal structures, we establish the global convergence of the generated iterates to a critical point. Experimental results on both grayscale and color image restoration demonstrate that the proposed method consistently outperforms representative model-based approaches while achieving performance comparable to state-of-the-art DEQ models based on implicit regularization, despite requiring substantially fewer trainable parameters.
\end{abstract}

\bigskip

\section{Introduction}
\label{intro}

Multiplicative Gamma noise commonly arises in image restoration problems associated with coherent imaging systems. Representative applications include Synthetic Aperture Radar (SAR), laser imaging, ultrasound imaging, and tomographic modalities, all of which involve coherent interference or exponential attenuation phenomena \cite{rudin2003multiplicative,el2026nonlocal,barnett2025multiscale}. Mathematically, the corresponding inverse problem aims to recover the unknown image $x$ from degraded observations
\begin{equation*}
	y=(A x) \eta,
\end{equation*}
where $A$ denotes a known ill-conditioned measurement operator and $\eta$ is a multidimensional Gamma random variable whose components have mean 1 and variance $\frac{1}{L}$, $L$ being the equivalent number of looks \cite{cui2011unsupervised}. The signal-dependent nature of the noise intensity, coupled with the complexity of the blurring operator, renders this inverse problem intrinsically ill-posed and challenging for both model formulation and algorithm design \cite{ran2025bv}.

Early studies on multiplicative noise removal were predominantly variational, incorporating prior image information through carefully designed regularization and fidelity terms. One of the earliest variational models for multiplicative denoising is the RLO model \cite{rudin2003multiplicative}, which combines total variation (TV) regularization with fidelity constraints on the image mean and variance. To better align the model with the statistical description of Gamma noise, Aubert and Aujol \cite{aubert2008variational} proposed a new fidelity term derived from a maximum a posteriori (MAP) perspective. The resulting model, however, is nonconvex, making convergence analysis considerably more challenging. Subsequent research introduced a variety of improvements, including variable transformation methods \cite{shi2008nonlinear}, variable splitting techniques \cite{huang2009new,zhao2014new}, and formulations involving auxiliary fidelity terms \cite{dong2013convex}. While TV-based regularization achieves satisfactory denoising performance, it is prone to artifacts such as contrast attenuation, corner smearing, and staircase effects \cite{zhu2012image}. Higher-order geometric regularization has since been proposed to mitigate these drawbacks: Zhang \cite{zhang2022image} introduced an adaptive Euler's elastica model within a level-set framework that suppresses staircase artifacts while preserving geometric structure, and Yang et al. \cite{yang2025mixed} developed a curvature-based elastica functional that further improves the preservation of edges and fine details.

Beyond variational models, nonlocal approaches have also been extensively investigated. Exploiting sparse representations among similar patches, Deledalle et al. \cite{deledalle2017mulog} developed the MuLoG+BM3D framework for speckle reduction in synthetic aperture radar images. Moreover, Liu et al. \cite{liu2020multiplicative} constructed a low-rank distance regularization based on nonlocal patch similarity and proposed a convergent optimization algorithm for the resulting model. For a more comprehensive overview of recent advances in multiplicative Gamma noise removal, the interested reader is referred to \cite{feng2021models,ullah2017new}.

Despite the aforementioned model‑based methods have shown promising performance, they rely heavily on handcrafted prior assumptions to characterize image structures through regularization. Consequently, when the prescribed priors fail to accurately capture the underlying image properties, the restoration quality may deteriorate significantly. Moreover, these methods typically require iterative stage-wise optimization and careful manual parameter tuning, which substantially increases the computational and practical burden.

Model-based deep learning (MBDL) methods have recently emerged as an effective alternative to purely model-driven approaches by integrating physical measurement models with data-driven learned priors. Among them, the plug-and-play (PnP) framework has attracted considerable attention due to its flexibility in incorporating learned Gaussian denoisers into a wide range of imaging tasks \cite{zhang2021plug,wang2019deep,li2023deep}. Another representative MBDL paradigm is deep unfolding (DU), which interprets each iteration of an optimization algorithm as a network layer and trains the resulting architecture in an end-to-end supervised manner \cite{chen2016trainable,zhang2018ista,muckley2021results}. Despite their success, DU methods are often limited by the memory consumption and numerical instability associated with backpropagation through deep iterative architectures. Consequently, practical implementations usually involve only a small number of unfolding iterations, which restricts flexibility during inference. To address this limitation, deep equilibrium (DEQ) models \cite{bai2019deep}, which can be interpreted as infinitely deep networks, have recently been introduced for solving imaging problems \cite{gilton2021deep}. Rather than explicitly stacking network layers, DEQ learns a fixed point whose equilibrium state approximates the ground-truth image while requiring substantially less memory during training. Under suitable smoothness assumptions on the learned operator together with standard optimization conditions, the convergence of the corresponding fixed-point iterations can be rigorously established. These theoretical guarantees have enabled successful applications to Poisson denoising \cite{daniele2026deep} and Gaussian image restoration \cite{gkillas2023optimization}. Nevertheless, extending the DEQ framework to multiplicative Gamma noise is substantially more involved, since the corresponding data-fidelity terms exhibit markedly stronger non-Lipschitz behavior, rendering convergence analyses based on the standard Lipschitz smoothness assumption no longer applicable.

In this work, we develop a deep equilibrium (DEQ) framework for image restoration under joint blur and multiplicative Gamma noise degradations. From a modeling perspective, we propose an explicit variational regularizer parameterized by learnable convolution filters and influence functions. The resulting regularizer naturally falls within the Field-of-Experts (FoE) framework \cite{roth2005fields} while explicitly incorporating geometric priors based on surface area and mean curvature, thereby combining the interpretability of variational models with the flexibility of data-driven learning. From an algorithmic perspective, we introduce a novel Bregman potential together with a mirror descent inference scheme tailored to Gamma noise. From a theoretical perspective, we establish global convergence of the inference iterations to a critical point of the underlying objective functional. Finally, extensive experiments on both grayscale and color image restoration demonstrate that the proposed method consistently outperforms representative variational and deep learning approaches.

The remainder of this paper is organized as follows. Section \ref{sec:DEQ_MD} introduces the parameterized variational model and the proposed mirror descent inference algorithm. Section \ref{sec:Cov_DEQ} presents the convergence analysis of the lower-level problem. Section \ref{sec:ForBack} describes the forward fixed-point computation and backward training strategy within the DEQ framework. Experimental results are reported in Section \ref{sec:Exper}. Finally, Section \ref{sec:conclude} draws some conclusions of the proposed studies. To keep the manuscript self-contained, the basic definitions and auxiliary theoretical results are provided in Appendices \ref{Appendix:Bregman} and \ref{Appendix:KLfunctions}.

\section{Deep equilibrium modeling with mirror descent}\label{sec:DEQ_MD}
Deep equilibrium models 
learn fixed points of parametric functions represented by neural network, and 
are equivalent to infinitely deep recurrent neural networks whose equilibrium states provide accurate approximations of the ground truth.
By drawing connections to classical optimization algorithms, such as gradient descent and proximal gradient methods, DEQ models can be naturally adapted to (linear) inverse problems with additive Gaussian noise \cite{gilton2021deep,gkillas2023highly,gkillas2023optimization}. However, in the presence of multiplicative Gamma noise, the commonly used data-fidelity terms are generally not $L$-smooth, which introduces additional difficulties in both algorithm design and convergence analysis. Motivated by \cite{zou2023deep,daniele2026deep}, we incorporate a mirror descent strategy into the DEQ framework to address image restoration problems involving blur operators and Gamma noise. Under suitable constraints on the optimization parameters, convergence guarantees of the proposed method can be established.

To systematically formulate the proposed DEQ framework, we consider a dataset $\{(x_i^*,y_i)\}_{i=1}^N$, where $x_i^*\in\mathbb{R}^n$ denotes the ground-truth image and $y_i\in\mathbb{R}^m$ represents the corresponding degraded observation. Let $f_\theta\left(\cdot;y_i\right):\mathbb{R}^n\rightarrow \mathbb{R}^n$ be an image-to-image neural network parameterized by $\theta\in\Theta\subset\mathbb{R}^p$. The proposed DEQ model is formulated as the bilevel problem:
\begin{equation}
\begin{aligned}
&\hat{\theta}\in \underset{\theta\in\Theta}{\arg\min}\;\mathcal{L}(\theta)=\frac{1}{N}\sum_{i=1}^N\ell\left(
x_i^\infty\big(\theta;(x_i^0,y_i)\big),x_i^*\right),\\
&\text{s.t.}\hspace{0.5cm} x_i^\infty\big(\theta;(x_i^0,y_i)\big)\in\text{Fix}(f_\theta(\cdot;y_i)).\\
\end{aligned}
\label{eq:bilevelproblm}
\end{equation}
Here, $x_i^0$ denotes the initial estimate for the $i$-th sample, $\ell$ is the upper-level loss function (e.g., the mean squared error), and $\operatorname{Fix}(f_\theta(\cdot;y_i))$ denotes the set of fixed points of the mapping $f_\theta(\cdot;y_i)$. Consequently, the lower-level problem amounts to computing a fixed point of $f_\theta(\cdot;y_i)$. To ensure consistency with the underlying variational model, we construct $f_\theta$ such that its fixed points coincide with the critical points of a regularized objective functional $F_\theta$. Specifically, under a box-constrained variational formulation \cite{chan2012multiplicative,zeng2019discontinuity}, we consider the composite energy
\begin{equation*}
F_\theta(x;y_i)=\mathcal{D}(x;y_i)+\lambda R_\theta(x)+\iota_{[0,a]^n}(x),
\end{equation*}
where $\mathcal{D}(x;y_i)$ denotes the data-fidelity term enforcing consistency between the reconstruction $x$ and the observed data $y_i$, $R_\theta$ is a learnable regularizer equipped with explicit geometric priors, $\lambda>0$ is the regularization parameter that balances the influence of the data term and the regularizer, and $\iota$ is the indicator function associated with the admissible set $[0,a]^n$, defined by
\begin{equation*}
\iota_{[0,a]^n}(x)=\begin{cases}
0, & x \in [0,a]^n, \\
+\infty, & \text{otherwise.}
\end{cases}
\end{equation*}
For imaging problems corrupted by multiplicative Gamma noise, the data-fidelity term $\mathcal{D}(x;y_i)$ is typically chosen as one of the following two formulations.

\begin{itemize}
\item[$\bullet$] (KL-type fidelity) A generalized Kullback--Leibler fidelity, originally introduced for Poisson noise and subsequently adapted to Gamma noise \cite{steidl2010removing,dong2013convex}:
\begin{equation*}
\mathcal{D}^{KL}(x;y_i)=\sum_{q=1}^m \left((Ax)_q-(y_i)_q\log(Ax)_q\right).
\end{equation*}
\item[$\bullet$] (AA-type fidelity) From the maximum a posteriori (MAP) perspective, the Aubert--Aujol fidelity term \cite{aubert2008variational} is given by:
\begin{equation*}
\mathcal{D}^{AA}(x;y_i)=\sum_{q=1}^m \left(\frac{(y_i)_q}{(Ax)_q}+\log(Ax)_q\right). 
\end{equation*}
\end{itemize}

In the following subsections, we introduce the parameterized regularization term $R_\theta$ and the fixed-point map $f_\theta$ associated with the minimization of $F_\theta$ employed in this work.

\subsection{Parameterized explicit variational regularization}
A variety of model-based neural networks built upon implicit architectures have recently been proposed for imaging inverse problems \cite{kadkhodaie2021stochastic,gilton2021deep}. Although these implicit frameworks demonstrate promising empirical performance, they are often treated as black-box models and lack rigorous mathematical interpretability. To improve the interpretability of the proposed framework, we explicitly incorporate geometric priors into the network parameterization. Such geometric structures have been shown to play a fundamental role in network design and image reconstruction; see, e.g., \cite{ye2019curvature,li2025curvpnp}. Inspired by the image area and mean curvature regularizers in \cite{yang2025mixed}, we consider the following parameterized regularization model equipped with learnable influence functions and linear convolution kernels:
\begin{equation}
\begin{aligned}
R_\theta(x)
=&\sum_{c=1}^C\sum_{j=1}^n
\Bigg(
    \sum_{l=1}^{N_a}
    \left[\psi_l^a\left(k_l*x^{(c)}\right)\right]_j
    \left[\mathcal A\!\left(x^{(c)}\right)\right]_j
\\
&\qquad
+b
\sum_{l'=1}^{N_c}
\left[\psi_{l'}^c\!\left(\mathcal K\!\big(x^{(c)}\big)\right)\right]_j
\left[\mathcal A\!\left(x^{(c)}\right)\right]_j
\Bigg).
\end{aligned}
\label{eq:regular1}
\end{equation}
Here, $C$ denotes the number of image channels, $N_a$ and $N_c$ denote the numbers of learnable area- and curvature-related influence functions, respectively, $k_l$ are learnable convolution kernels, $*$ is the convolution operator, and $b$ is a learnable scalar. Let $D_h^+(D_h^-)$ and $D_v^+(D_v^-)$ denote the forward (backward) finite difference operators along the horizontal and vertical directions, respectively. Based on these operators, we define the following two geometric quantities associated with $x^{(c)}$.
\begin{enumerate}[label=(\roman*)]
\item The area-related geometric quantity is defined by
\begin{equation*}
\left[\mathcal{A}\left(x^{(c)}\right)\right]_j=\sqrt{\left[D^+_h x^{(c)}\right]_j^2+\left[D^+_v x^{(c)}\right]_j^2+\epsilon},
\end{equation*}
where $\epsilon$ is a learnable positive parameter.
\item The curvature-related geometric quantity is defined as
\begin{equation*}
\mathcal{K}\big(x^{(c)}\big)=D_h^-\frac{D^+_h x^{(c)}}{\mathcal{A}(x^{(c)})}+ D_v^-\frac{D^+_v x^{(c)}}{\mathcal{A}(x^{(c)})},
\end{equation*}
where the division is understood element-wise.
\end{enumerate}

The functions $\psi_l^a \colon \mathbb{R}^n \to \mathbb{R}^n$ and $\psi_{l^\prime}^c\colon \mathbb{R}^n \to \mathbb{R}^n$ denote learnable influence functions associated with the area and curvature terms, respectively. Since the parameterizations of the two influence functions are analogous, we only describe the construction of $\psi_l^a$.
Following the strategy of \cite{chen2016trainable}, the influence function $\psi_l^a$ 
acts separately on each component and is described using a radial basis function (RBF) expansion as follows
\begin{equation*}
\left[\psi_l^a(z)\right]_j=\sum_{k=1}^M\omega_{lk}^a\phi\left(\frac{|z_j-\mu_k|}{\gamma_k}\right),
\end{equation*}
where $\omega_{lk}^a$ are learnable coefficients, $\mu_k$ are uniformly spaced centers, and $\gamma_k$ are the corresponding scaling parameters. The learnable coefficients associated with $\psi_{l^\prime}^c$ are denoted by $\omega_{l^\prime k}^c$. The basis function $\phi$ is chosen as a Gaussian one,
\begin{equation*}
\phi\left(\frac{|z_j-\mu_k|}{\gamma_k}\right)=\exp\left(-\frac{\left(z_j-\mu_k\right)^2}{2\gamma_k^2}\right).
\end{equation*}
This choice guarantees the analyticity of the learned influence functions and the resulting regularizer $R_\theta$, which is required for the convergence analysis in Proposition \ref{pro:PsiKL}.

Finally, each convolution kernel is parameterized using a discrete cosine transform (DCT) basis \cite{chen2016trainable},
\begin{equation*}
k_l=\sum_{r=1}^{N_r}\frac{w_{lr}b_r}{\|w_{l}\|},
\end{equation*}
where $b_r \in \mathbb{R}^{N_r}$ are fixed DCT basis filters and $w_{l}$ is a vector of $N_r$ learnable coefficients. Under this parameterization, the resulting kernels naturally satisfy the unit-norm property and characterize the linear component of the regularization model, while the influence functions capture the corresponding nonlinear behavior. Consequently, the learnable parameter set $\theta=\{\omega_{lk}^a,\omega_{l^\prime k}^c,w_{lr},\epsilon,b\}$ endows the proposed regularization model with considerable flexibility and expressive power.

Notably, this flexibility is not merely qualitative: by appropriately specializing $\theta$, the regularizer $R_\theta$ recovers a range of classical variational models as special cases, indicating that it constitutes a genuine generalization rather than an ad hoc parameterization. To illustrate this property, we present several representative special cases in the single-channel setting with $N_a=N_c=1$.
\begin{itemize}
\item[$\bullet$]\textbf{Total Variation (TV).}

If 
\begin{equation*}
[\psi_l^a(\cdot)]_j\equiv 1,\ \forall j, \;\epsilon=0, \; b=0, 
\end{equation*}
then \eqref{eq:regular1} reduces to the classical total variation (TV) regularizer \cite{rudin2003multiplicative}.

\item[$\bullet$]\textbf{Minimal Surface regularization.}

If
\begin{equation*}
[\psi_l^a(\cdot)]_j\equiv 1, \  \forall j, \;\epsilon=1, \; b=0, 
\end{equation*}
then it becomes the minimal surface regularizer \cite{pang2018image}.
\item[$\bullet$]\textbf{MC-$L^p$.}

If
\begin{equation*}
[\psi_l^a(\cdot)]_j\equiv 0, \  \forall j, \;[\psi^c_{l^\prime}(z)]_j = |z_j|^p,\; \epsilon=1,
\end{equation*}
where $|\cdot|$ denotes the absolute value, with fixed constants $p$ and $b$, then \eqref{eq:regular1} reduces to the MC-$L^p$ regularizer proposed in \cite{zhu2020image}.
\item[$\bullet$]\textbf{Euler's elastica regularization.}

If
\begin{equation*}
k_l=G_\sigma,\quad
[\psi^a_l(z)]_j = \left(\frac{z_j}{\displaystyle \max_j z_j}\right)^p,\
[\psi^c_{l^\prime}(z)]_j = (z_j)^2,\
\epsilon=0,
\end{equation*}
where $G_\sigma$ denotes the Gaussian convolution kernel with standard deviation $\sigma$, and $b$ is a fixed constant, then \eqref{eq:regular1} reduces to the Euler's elastica regularizer proposed in \cite{zhang2022image}. With these choices, the first regularization term in \eqref{eq:regular1} provides an estimate of the gray level of the (degraded) image, which scales the intensity of the regularization according to the intensity of each pixel. This is a desirable effect when dealing with multiplicative noise. 
\item[$\bullet$]\textbf{Hybrid geometric regularization.}

If
\begin{equation*}
k_l=G_\sigma,\quad
[\psi^a_l(z)]_j = \left(\frac{z_j}{\displaystyle \max_j z_j}\right)^p,\
[\psi^c_{l^\prime}(z)]_j = (z_j)^2,\
\epsilon=1,
\end{equation*}
then the resulting regularizer coincides with the hybrid geometric regularization proposed in \cite{yang2025mixed}.
\end{itemize}
\subsection{Mirror descent algorithms for deep equilibrium models}
\label{ssec:MD_DEQ}
This section derives the fixed-point mapping $f_\theta$ associated with the minimization of the parametric energy $F_\theta$. In convex optimization, it is common to rely on iterative schemes to progressively refine an approximation of a minimizer: as a consequence, such minimizers can be seen as fixed points of the map employed to produce the next iterate given the previous one.
For the multiplicative Gamma noise restoration problems considered in this work, the AA-type fidelity $\mathcal{D}^{AA}(x;y_i)$ is not globally convex, while $\mathcal{D}^{AA}(x;y_i)$ and $\mathcal{D}^{KL}(x;y_i)$ both fail to satisfy the standard $L$-smoothness condition. Consequently, classical optimization methods, such as proximal gradient descent, cannot be directly applied with rigorous convergence guarantees. Motivated by \cite{daniele2026deep}, we adopt a mirror descent framework to address these difficulties. By introducing a Bregman potential function $h$, mirror descent extends the standard proximal gradient method and relaxes the conventional smoothness and convexity requirements imposed on the objective functional.

To present the mirror descent framework in a general setting, we consider the following nonconvex optimization problem:
\begin{equation*}
\inf_{x\in\overline{C}}\Psi(x):=f(x)+g(x),
\end{equation*}
where $C = \operatorname{int} \operatorname{dom} h$ and $h\in\mathcal{G}(C)$ is a Bregman potential satisfying the kernel generating distance property (see Definition \ref{def:kernelgere} for the precise definition). To establish the well-posedness of the subsequent algorithm in the considered nonconvex case, the following assumptions are needed (see \cite{bolte2018first}):
\begin{assumption}
\leavevmode\\[-0.8\baselineskip]
\begin{enumerate}[label=(\roman*),font=\normalfont]
\item $h\in\mathcal{G}(C)$ and $\overline{C}=\overline{\operatorname{dom} h}$.
\item $f:\mathbb{R}^n\to(-\infty,+\infty]$ is proper, lower semicontinuous and convex, with
$\operatorname{dom} f\cap C\neq\emptyset$.
\item $g:\mathbb{R}^n\to(-\infty,+\infty]$ is proper and lower semicontinuous such that
$\operatorname{dom} h\subseteq \operatorname{dom} g$. Moreover, $g$ is continuously differentiable on $C$.
\item The objective function $\Psi$ is bounded below on $\overline{C}$, i.e.,
\[
\inf_{x\in\overline{C}}\Psi(x)>-\infty .
\]
\item For any $\theta>0$, the function $h(x)+\theta f(x)$ is supercoercive, i.e.,
\[
\lim_{\|x\|\rightarrow +\infty}\frac{h(x)+\theta f(x)}{\|x\|}=+\infty.
\]
\end{enumerate}
\label{ass:fandg}
\end{assumption}

Starting from an initialization $x^0\in\operatorname{int}\left(\operatorname{dom}h\right)$, the mirror descent iteration is recursively updated by
\begin{equation}
x^{k+1}=\underset{x \in \mathbb{R}^n}{\arg\min}\Big\{\Psi_h(x):=f(x)+\langle x-x^k,\nabla g(x^k)\rangle+\frac{1}{\tau}D_h(x,x^k)\Big\}.
\label{eq:xkupdate}
\end{equation}
Here, $\tau>0$ denotes the step size and $D_h$ is the Bregman distance induced by $h$; see Definition \ref{def:bregmandistance} for its definition. The update can be interpreted as minimizing a local model obtained by linearizing the differentiable component of the objective fuction $g$ at the current iterate and regularizing the resulting subproblem through the non-Euclidean metric $D_h$.
\remark When $h=\|\cdot\|^2$, the above scheme reduces to the classical proximal gradient method. Hence, mirror descent may be viewed as a non-Euclidean generalization of proximal gradient descent \cite{bauschke2017descent}.
\remark Although the objective function $\Psi(x)=f(x)+g(x)$ is generally nonconvex due to the nonconvexity of $g$, the local surrogate function $\Psi_h$ arising from the linearization of $g$ becomes convex under Assumption \ref{ass:fandg}. Consequently, the resulting update mapping from $\operatorname{int}(\operatorname{dom}h)$ to itself is well defined and single valued.

Returning to the restoration problem considered in this work, we set $f$ as the indicator function $\iota_{[0,a]^n}(x)$, while $g$ consists of the parameterized regularization network $\lambda R_\theta(x)$ and the data-fidelity term $\mathcal{D}(x;y)$, with domain $C=(0,+\infty)$. Consequently, the generic objective $\Psi$ introduced above reduces to the parameterized functional $F_\theta$ employed in the proposed DEQ model. Under Assumption \ref{ass:fandg}, iteration \eqref{eq:xkupdate} admits the following closed-form expression as the lower-level subproblem of the DEQ model:
\begin{equation}
x^{k+1}=\operatorname{Prox}_{\tau \iota_{[0,a]^n}}^h\left(\nabla h^*\left(\nabla h(x^k)-\tau\left(\lambda\nabla R_\theta(x^k)+\nabla \mathcal{D}(x^k;y)\right)\right)\right)
\label{eq:MD}
\end{equation}
where $h^*$ denotes the Fenchel conjugate of $h$ \cite{rockafellar1997convex}. The choice of the Bregman potential $h$ depends on the specific form of the fidelity term. For the KL-type fidelity, a natural choice is Burg's entropy \cite{bauschke2017descent},
\begin{equation*}
h^{KL}(x)=-\sum_{j=1}^n\log(x_j).
\end{equation*}
which inherently enforces the positivity constraint on the solution. In contrast, for the AA-type fidelity term, this choice becomes unsuitable since it fails to satisfy the smoothness conditions required in the subsequent convergence analysis; see Lemma \ref{lema:1divxLhg}. In this case, we instead choose the potential function
\begin{equation*}
h^{AA}(x)=\sum_{j=1}^n\frac{1}{x_j}.
\end{equation*}
Although this potential $h^{AA}(x)$ does not explicitly preserve positivity, the constraint is still effectively enforced through the indicator function $\iota$ associated with the admissible set $[0,a]^n$, making it an appropriate choice within the proposed mirror descent framework.

As a particular instance of the Bregman proximal gradient framework, the Bregman proximal step in mirror descent is generally not equivalent to an Euclidean projection. Nevertheless, under the specific choices of the potential functions $h^{KL}(x)$ and $h^{AA}(x)$, corresponding respectively to the KL-type and AA-type fidelity terms, this equivalence still holds.
\begin{proposition} \label{prop:projection}
Let $h(x)=\sum_{j=1}^n x_j^{-1}$. Then, for any $x\in\mathbb{R}_{++}^n$, the Bregman proximal operator associated with the indicator function $\iota_{[0,a]^n}$ coincides with the Euclidean projection onto $[0,a]^n$.
\end{proposition}

\begin{proof}
The Bregman proximal operator associated with $\iota_{[0,a]^n}$ is defined by
\begin{equation}
\begin{aligned}
\mathrm{Prox}_{\tau\iota_{[0,a]^n}}^h(x)
&=\arg\min_{u\in\mathbb{R}^n}
\left\{
\tau\iota_{[0,a]^n}(u)+D_h(u,x)
\right\} \\
&=\arg\min_{u\in\mathbb{R}^n}
\sum_{j=1}^n
\left[
\tau\iota_{[0,a]}(u_j)
+\left(
\frac{1}{u_j}
-\frac{1}{x_j}
+\frac{u_j-x_j}{x_j^2}
\right)
\right].
\end{aligned}
\label{eq:prb_Proj_Prox}
\end{equation}
Since the objective function is separable, the minimization decouples componentwise. Define
\[
\tilde{D}_h(\tilde{u},\tilde{x})
=
\frac{1}{\tilde{u}}
-\frac{1}{\tilde{x}}
+\frac{\tilde{u}-\tilde{x}}{\tilde{x}^2},
\qquad \tilde{x}>0.
\]
Then \eqref{eq:prb_Proj_Prox} reduces to
\[
\mathrm{Prox}_{\tau\iota_{[0,a]^n}}^h(x)
=
\left(
\mathrm{Prox}_{\tau\iota_{[0,a]}}^h(x_1),
\ldots,
\mathrm{Prox}_{\tau\iota_{[0,a]}}^h(x_n)
\right),
\]
where
\[
\mathrm{Prox}_{\tau\iota_{[0,a]}}^h(x_j)
=
\arg\min_{0\le u_j\le a}
\tilde{D}_h(u_j,x_j).
\]
Differentiating $\tilde{D}_h$ with respect to $\tilde{u}$ yields
\[
\partial_{\tilde{u}}\tilde{D}_h(\tilde{u},\tilde{x})
=
\frac{\tilde{u}^2-\tilde{x}^2}
{\tilde{u}^2\tilde{x}^2}.
\]
Hence, $\tilde{D}_h(\cdot,\tilde{x})$ is strictly decreasing on $(0,\tilde{x})$ and strictly increasing on $(\tilde{x},\infty)$, so its unique minimizer is attained at $\tilde{u}=\tilde{x}$. Therefore,
\[
\mathrm{Prox}_{\tau\iota_{[0,a]}}^h(x_j)
=
\min\{x_j,a\},
\]
which is precisely the Euclidean projection onto $[0,a]$. Consequently,
\[
\mathrm{Prox}_{\tau\iota_{[0,a]^n}}^h(x)
=
\Pi_{[0,a]^n}(x),
\]
where $\Pi_{[0,a]^n}$ denotes the Euclidean projection onto $[0,a]^n$.
$\hfill\Box$
\end{proof}

We conclude this section by presenting the explicit mirror descent updates for image restoration under multiplicative Gamma noise. The resulting closed-form iterations are derived for the two fidelity terms most commonly employed in applications.

For the KL-type fidelity term, the corresponding update is given by
\begin{equation}
x^{k+1}=\Pi_{[0,a]^n}\left(\frac{x^k}{1+\tau x^k\left(\lambda\nabla R_\theta(x^k)+\nabla \mathcal{D}^{KL}(x^k;y)\right)}\right).
\label{eq:MD_KL}
\end{equation}

For the AA-type fidelity term, the mirror descent iteration takes the form
\begin{equation}
x^{k+1}=\Pi_{[0,a]^n}\left(\frac{x^k}{\sqrt{1+\tau (x^k)^2\left(\lambda\nabla R_\theta(x^k)+\nabla \mathcal{D}^{AA}(x^k;y)\right)}}\right).
\label{eq:MD_AA}
\end{equation}
These closed-form expressions constitute the computational core of the proposed method and will serve as the forward operator in the subsequent DEQ framework.

\begin{remark}
The update formulas \eqref{eq:MD_KL} and \eqref{eq:MD_AA} are derived under the assumption that the argument of the projection operator $\Pi_{[0,a]^n}$ belongs to $\mathbb{R}^n_{++}$. To ensure that this assumption is satisfied throughout the iterative process, we introduce in Section \ref{sec:ForBack} a backtracking strategy that preserves the positivity of the projection argument.
\end{remark}

\section{Convergence analysis of mirror descent with smooth-adaptable functionals}\label{sec:Cov_DEQ}
In this section, we establish the convergence of the mirror descent iterations to a critical point of the objective functional $\Psi$, where $\Psi=F_\theta$. A common assumption in the convergence analysis of first-order optimization methods is the Lipschitz continuity of the gradient; see, e.g., \cite{attouch2013convergence,bolte2014proximal}. However, this assumption is violated by both the AA-type and KL-type data-fidelity terms arising in multiplicative Gamma noise restoration. 
Moreover, the objective functional $\Psi$ is nonconvex, so classical convex optimization theory is not applicable. To overcome these difficulties, we follow the analytical framework of \cite{bolte2018first,daniele2026deep} and establish the convergence of the mirror descent scheme \eqref{eq:MD} under weaker assumptions satisfied by the proposed model. Specifically, the classical Lipschitz smoothness requirement is replaced by the notion of L-smooth adaptability with respect to the Bregman potential $h$, while the convexity assumption is replaced by the Kurdyka--\L{}ojasiewicz property.

Since the convergence analysis for the KL-type fidelity follows directly from the arguments in \cite{daniele2026deep}, we restrict our attention to the AA-type fidelity, which requires additional technical developments.

We begin by recalling the notion of L-smooth adaptability ($L$-smad).
\definition Let $h\in\mathcal{G}(C)$ with $C=\operatorname{int}(\operatorname{dom}h)$, and let $g:\mathbb{R}^n\rightarrow (-\infty,+\infty]$ be a proper lower semicontinuous function such that $\operatorname{dom}h\subset\operatorname{dom}g$. Assume that $g$ is continuously differentiable on $C$. The pair $(g,h)$ is said to satisfy the L-smooth adaptability ($L$-smad) condition on $C$ if:

\centering{there exists $L>0$ such that $Lh-g$ is convex on $C$.}
\vspace{0.2cm}
\begin{remark}
Strictly speaking, the complete $L$-smad condition additionally requires $Lh+g$ to be convex. However, for the convergence analysis considered here, it is sufficient to assume only the convexity of $Lh-g$ \cite{bolte2018first}. 
\end{remark}

\justifying
Let $a_q \in\mathbb{R}_+^n$ denote the $q$-th row of the forward operator $A \in \mathbb{R}^{m \times n}$. Assume that $a_q\neq \boldsymbol{0}$ for all $q=1, \ldots, m$. Then, for every $x\in\mathbb{R}_{++}^n$, 
\[
\langle a_q,x\rangle>0, \quad q=1,\ldots,m.
\]
This is a standard assumption for blur operators in imaging inverse problems; see, e.g., \cite{dong2013convex,bertero2009image}. Under this condition, one can derive a mild sufficient criterion ensuring the $L$-smad property for the AA-type fidelity term.
\begin{lemma}\label{lema:1divxLhg}
Let $g(x)=\mathcal{D}^{AA}(x;y)$ and $h^{AA}(x)=\sum_{j=1}^n\frac{1}{x_j}$. Then, for any
\begin{equation*} 
L\geq M\sum_{q=1}^m y_q,
\end{equation*}
where 
\begin{equation*}
M=\max_{1\leq q\leq m}\frac{\max_{1\leq j\leq n}\{a_{qj}^2\}}{\min_{1\leq j\leq n}\{a_{qj}^3,a_{qj}>0\}},
\end{equation*}
the function $Lh^{AA}-g$ is convex on $\mathbb{R}_{++}^n$.
\end{lemma}
\begin{proof}
Since $g$ and $h^{AA}$ are $\mathcal{C}^2$ on $\mathbb{R}_{++}^n$, it suffices to find a constant $L>0$ such that, for any $x\in\mathbb{R}_{++}^n$ and any direction $d\in\mathbb{R}^n$,
\begin{equation*}
L\langle\nabla^2 h^{AA}(x)d,d\rangle \geq \langle\nabla^2 g(x)d,d\rangle.
\end{equation*}
Firstly, it is straightforward to verify that, for any $d\in\mathbb{R}^n$,
\begin{equation*}
L\langle\nabla^2 h^{AA}(x)d,d\rangle=2L\sum_{j=1}^n\frac{d_j^2}{x_j^3}.
\end{equation*}
On the other hand, by the definition of $g$, we have
\begin{equation*}
\nabla g(x)=\sum_{q=1}^m\frac{1}{\langle a_q,x\rangle}a_q-\frac{y_q}{\langle a_q,x\rangle^2}a_q,
\end{equation*}
and
\begin{equation*}
\langle\nabla^2 g(x)d,d\rangle=\sum_{q=1}^m-\frac{\langle a_q,d\rangle^2}{\langle a_q,x\rangle^2}+\frac{2y_q\langle a_q,d\rangle^2}{\langle a_q,x\rangle^3}\leq \sum_{q=1}^m 2y_q\frac{\langle a_q,d\rangle^2}{\langle a_q,x\rangle^3}.
\end{equation*}
By the Cauchy-Schwarz inequality and the assumption that $a_q\in\mathbb{R}_+^n$ with $a_q\neq \boldsymbol0$ for all $q=1,\ldots,m$, we obtain
\begin{equation*}
\frac{\langle a_q,d\rangle^2}{\langle a_q,x\rangle^3}\leq\frac{\sum_{j=1}^na_{qj}^2x_j^3}{\langle a_q,x\rangle^3}\sum_{j=1}^n\frac{d_j^2}{x_j^3}.
\end{equation*}
Let $M_q=\frac{\max_{j}\{a_{qj}^2\}}{\min_j\{a_{qj}^3:a_{qj}>0\}}$. Then it follows that
\begin{equation*}
\frac{\sum_{j=1}^na_{qj}^2x_j^3}{\langle a_q,x\rangle^3}\leq \frac{\max_{j}\{a_{qj}^2\}\sum_{j=1}^nx_j^3}{\min_j\{a_{qj}^3:a_{qj}>0\}\sum_{j=1}^n(x_j)}\leq M_q.
\end{equation*}
Consequently, we have
\begin{equation*}
\begin{aligned}
\langle\nabla^2 g(x)d,d\rangle\leq& \sum_{q=1}^m2y_q\frac{\langle a_q,d\rangle^2}{\langle a_q,x\rangle^3}\\
\leq&\sum_{q=1}^m 2y_qM_q\left(\sum_{j=1}^n\frac{d_j^2}{x_j^3}\right)\\
\leq&M\left(\sum_{q=1}^m 2y_q\right)\left(\sum_{j=1}^n\frac{d_j^2}{x_j^3}\right),
\end{aligned}
\end{equation*}
where $M=\max_{q\in\{1,\ldots,m\}}M_q>0$ is a constant that depends only on the matrix $A$ and is independent of $x$ and $d$.

Therefore it suffices to choose
\begin{equation*}
L\geq M\sum_{q=1}^m y_q,
\end{equation*}
which guarantees the convexity of $Lh-g$.
$\hfill\Box$
\end{proof}
\begin{remark}
Burg's entropy is the standard choice of Bregman potential for the KL-type fidelity and plays a central role in the existing convergence theory. It is therefore natural to consider using the same potential for the AA-type fidelity. Nevertheless, the corresponding $L$-smad condition generally fails because the AA-type fidelity possesses a stronger singularity near the origin, induced by its $1/x$-type structure. The proposed potential can thus be viewed as an extension of Burg's entropy tailored to the stronger singular behavior of the AA-type fidelity, while retaining the analytical properties required for our convergence analysis.
\end{remark}

For the KL-type fidelity term, the $L$-smad condition associated with the Bregman potential $h^{KL}$ admits a considerably simpler characterization than in the AA case. In fact, the corresponding result is classical; see \cite[Lemma 7]{bauschke2017descent}, which provides an explicit condition on $L$ ensuring the convexity of $Lh-g$.

\begin{lemma}\label{lema:BurgLhg}
Let $g(x)=\mathcal{D}^{KL}(x;y)$, and let $h^{KL}(x)$ denote the Burg's entropy. If the constant $L$ satisfies
\begin{equation*}
L\geq \sum_{q=1}^m y_q,
\end{equation*}
then the function $Lh-g$ is convex on $\mathbb{R}^n_{++}$.
\end{lemma}
Based on Lemmas \ref{lema:1divxLhg} and \ref{lema:BurgLhg}, establishing the $L$-smad property for the objective function $\mathcal{D}(x;y)+\lambda R_\theta$ reduces to verifying 
it for the pair
$(\lambda R_\theta,h)$,
since the sum of convex functions remains convex. As observed in \cite{hurault2023convergent}, verifying the $L$-smad condition over the entire domain $\operatorname{int}(\operatorname{dom}h)$ is generally difficult. A practical simplification is to restrict the analysis to a compact subset of $\operatorname{conv}(\operatorname{dom}R_\theta)\cap\operatorname{int}(\operatorname{dom}h)$. 

Motivated by this observation, we introduce the indicator function $\iota_{[0,a]^n}$ to constrain the admissible set to $[0,a]^n$. Such a restriction is natural in imaging applications, where image intensities are typically normalized to ranges such as $[0,1]^n$ or $[0,255]^n$. Under this bounded-domain setting, the selection rule for the constant $L^\prime$ ensuring the $L$-smad property of $(\lambda R_\theta,h)$ can be provided.
\begin{lemma}[$L$-smad property of $(\lambda R_\theta,h)$]\label{lemma:rtheta_smooth}
Let $R_\theta:(0,a]^n\rightarrow\mathbb{R}$ be a real analytic function and let
$
h(x)=\sum_{j=1}^n\frac{1}{x_j}$.
Then there exists a constant $L^\prime>0$ such that the pair $(\lambda R_\theta,h)$ satisfies the $L$-smad property on $(0,a]^n$. In particular, letting $\rho(\cdot)$ denote the spectral radius, $M:=\sup_{x\in(0,a]^n}\rho\!\left(\nabla^2R_\theta(x)\right)<\infty$,
then any constant
$
L^\prime\geq \frac{a^3}{2}\lambda M
$
is sufficient.
\end{lemma}
\begin{proof}
Since both $R_\theta$ and $h$ are both $\mathcal{C}^2$-functions on their domains, the $L$-smad property of the pair $(\lambda R_\theta,h)$ can be characterized through their Hessian matrices. In particular, it suffices to show that there exists $L^\prime>0$ such that \cite{bauschke2017descent}
\begin{equation*}
L^\prime\nabla^2h(x)-\lambda\nabla^2R_\theta(x)\succeq 0, \quad\forall x\in(0,a]^n.
\end{equation*}
where $\succeq 0$ denotes positive semidefiniteness. Equivalently, one seeks $L^\prime>0$ satisfying
\begin{equation*}
\inf_{x\in(0,a]^n}\lambda_{\min}\left(L^\prime\nabla^2h(x)-\lambda\nabla^2R_\theta(x)\right)\geq 0,
\end{equation*}
with $\lambda_{\min}(\cdot)$ the minimum eigenvalue. Therefore, following the argument of \cite[Lemma B.4]{daniele2026deep}, it remains to establish a positive lower bound for the eigenvalues of $\nabla^2h(x)$ and a uniform upper bound for those of $\nabla^2R_\theta(x)$ on $(0,a]^n$.

Since $R_\theta$ is analytic, it is Lipschitz smooth on the bounded set $(0,a]^n\subset [0,a]^n$. Consequently, there exists a constant $M>0$ such that
\begin{equation*}
\sup_{x\in(0,a]^n}\rho\left(\nabla^2R_\theta(x)\right)\leq M,
\end{equation*}
where $\rho(\cdot)$ denotes the spectral radius. On the other hand, consider the potential function $h(x)=h^{AA}(x)=\sum_{j=1}^n\frac{1}{x_j}$. For $x\in(0,a]^n$, we have
\begin{equation*}
\nabla^2h(x)=2\text{diag}\left(\frac{1}{x^3}\right), \quad \text{with}\; \frac{1}{x^3}=\left(\frac{1}{x_1^3},\ldots,\frac{1}{x_n^3}\right).
\end{equation*}
Hence, the Hessian matrix of $h$ satisfies
\begin{equation*}
\inf_{x\in(0,a]^n}\lambda_{\min}\left(\nabla^2h(x)\right)\geq\inf_{x\in(0,a]}\frac{2}{x^3}\geq \frac{2}{a^3}.
\end{equation*}
Therefore, whenever $L^\prime\geq \frac{a^3}{2}\lambda M>0$, we obtain
\begin{equation*}
\inf_{x\in(0,a]^n}\lambda_{\min}\left(L^\prime\nabla^2h(x)-\lambda\nabla^2R_\theta(x)\right)\geq 0,
\end{equation*}
This proves that the pair $(\lambda R_\theta,h)$ satisfies the $L$-smad condition on $(0,a]^n$.
$\hfill\Box$
\end{proof}
Finally, combining Lemmas \ref{lema:1divxLhg} and \ref{lemma:rtheta_smooth}, we obtain an estimate of the $L$-smad constant for the objective functional $F=\mathcal{D}^{AA}+\lambda R_\theta$. In particular, an admissible choice is $\bar{L}=L+L'$, yielding
\begin{equation}
    \bar{L} h^{AA} - F \quad \text{is convex on } \operatorname{int}(\operatorname{dom} h^{AA}).
    \label{eq:Lsmad_AA}
\end{equation}
Hence, the objective functional satisfies the $L$-smooth adaptability condition.

To establish the convergence of the mirror descent algorithm in this nonconvex setting, in addition to the $L$-smad property, we also rely on the Kurdyka-\L{}ojasiewicz (K\L{}) property of the minimized functional. Directly verifying the K\L{} property from its definition is generally challenging. Fortunately, a broad class of functions arising in optimization and imaging are known to satisfy the K\L{} property through $o$-minimal geometry. A detailed review of K\L{} functions and $o$-minimal structures is reported in Appendix \ref{Appendix:KLfunctions}. We therefore establish the K\L{} property of the objective functional $\Psi$ by showing that it is definable in an $o$-minimal structure, which in turn guarantees the desired K\L{} property and provides the remaining key ingredient for proving convergence to a critical point.
\begin{proposition}
Let $\Psi=\mathcal{D}^{AA}(x;y)+\iota_{[0,a]^n}(x)+\lambda R_\theta(x)$. Then $\Psi$ satisfies the K\L{} property at any $x\in\operatorname{crit}(\Psi)$. Moreover, $\Psi$ is coercive, and the set
\[
S_\infty=\{x\in\overline{\operatorname{dom}\Psi}:\Psi(x)=+\infty\}
\]
is closed.
\label{pro:PsiKL}
\end{proposition}
\begin{proof}
Following \cite[Proposition 4.3]{daniele2026deep}, it suffices to verify the properness, lower semicontinuity, continuity on the domain, and definability of $\Psi$ to establish the K\L{} property.

\noindent\textbf{1. $\Psi$ is proper, lower semicontinuous, and continuous on its domain.}
Indeed, for any $x\in (0,a)^n$, one has $\Psi(x)<+\infty$, implying that $\Psi$ is proper. Since $[0,a]^n$ is closed, the indicator function $\iota_{[0,a]^n}$ is lower semicontinuous \cite{rockafellar1997convex}. Consequently, $\Psi=\mathcal{D}^{AA}(x;b)+\iota_{[0,a]^n}(x)+\lambda R_\theta(x)$ being the sum of lower semicontinuous functions, is itself lower semicontinuous. Moreover, for $x\in\operatorname{dom}(\Psi)$, one has $\Psi=\mathcal{D}^{AA}(x;y)+\lambda R_\theta(x)$, and continuity follows directly from the continuity of $\mathcal{D}^{AA}$ and $R_\theta$.

\vspace{0.3cm}
\noindent\textbf{2. $\Psi$ is a definable fuction in $\mathbb{R}_{\text{an,exp}}$.}

We first recall the AA-type fidelity term
\begin{equation*}
\mathcal{D}^{AA}(x;y)=\sum_{q=1}^m \frac{y_q}{(Ax)_q}+\log(Ax)_q=\sum_{q=1}^m\mathcal{D}^{AA}_q(x). 
\end{equation*}
To establish its definability, it suffices to consider the elementary function $f(x)=1/x$ on the domain $\{x>0\}$. Its graph is given by
\begin{equation*}
\Gamma(f)=\{(x,t)\in\mathbb{R}^2:tx-1=0,x>0\},
\end{equation*}
which is a semialgebraic set. Hence, $f$ is definable in the semialgebraic structure $(\mathbb{R},+,\cdot,(r)_{r\in \mathbb{R}})$ and, consequently, definable in the larger $o$-minimal structure $\mathbb{R}_{\text{an,exp}}$. Since the logarithm is also definable in $\mathbb{R}_{\text{an,exp}}$, each component in $\mathcal{D}^{AA}_q(x)$ , being the sum of a logarithmic term and a reciprocal term, is definable. By closure of definable functions under finite sums, the fidelity term $\mathcal{D}^{AA}$ is therefore a definable function in $\mathcal{\mathbb{R}_{\text{an,exp}}}$. 

We next show that $\iota_{[0,a]^n}$ is semialgebraic, and hence definable in $\mathbb{R}_{\text{an,exp}}$. Consider the polynomials
\begin{equation*}
P^0_j(x)=0,
P^1_j(x)=-1,
P^2_j(x)=x_j(x_j-a),  
P^3_j(x)=x_j, 
P^4_j(x)=x_j-a, \forall j=1\ldots,n.
\end{equation*}
Then the cube $[0,a]^n$ can be expressed as a finite union of finite intersections of semialgebraic sets:
\begin{equation*}
\begin{aligned}
[0,a]^n = \bigcup_{\substack{S_0,S_a\subseteq\{1,\dots,n\} \\ S_0 \cap S_a = \varnothing}} \Big( &\bigcap_{j \in S_0} \{x \in \mathbb{R}^n : P^3_j(x)=0,P^1_j(x)<0\} \\
&\cap \bigcap_{j \in S_a} \{x \in \mathbb{R}^n: P^4_j(x)=0,P^1_j(x)<0\} \\
&\cap \bigcap_{j \notin S_0\cup S_a} \{x \in \mathbb{R}^n : P^0_j(x)=0,P^2_j(x) < 0\} \Big).
\end{aligned}
\end{equation*}
Hence $[0,a]^n$ is a semialgebraic set, and its indicator function $\iota_{[0,a]^n}$ is a semialgebraic function.

The network $R_\theta$ consists of compositions of affine mappings, convolution operators, and nonlinear activation functions constructed from Gaussian basis expansions. Therefore, $R_\theta$ is definable in $\mathbb{R}_{\text{an,exp}}$. 

Since all three components of $\Psi$ are definable in the $o$-minimal structure $\mathbb{R}_{\text{an,exp}}$, the objective functional $\Psi$ is also definable in $\mathbb{R}_{\text{an,exp}}$. Consequently, Theorem \ref{theo:ominKL} implies that $\Psi$ satisfies the Kurdyka-\L{}ojasiewicz property.

\noindent\textbf{3. $\Psi$ is coercive.}

Owing to the presence of the indicator function $\iota_{[0,a]^n}$, one has $\Psi(x)=+\infty$ whenever $x\notin [0,a]^n$. Hence
\begin{equation*}
\|x\| \rightarrow +\infty \quad \Rightarrow \quad \Psi(x) \rightarrow +\infty,
\end{equation*}
which establishes coercivity.

\vspace{0.3cm}
\noindent\textbf{4. $S_\infty$ is closed.}

In the present setting, $\overline{\operatorname{dom}(\Psi)}=[0,a]^n$. Moreover, for $x\in\overline{\operatorname{dom}(\Psi)}$, the condition $x\in S_\infty$ is equivalent to the existence of some $q\in\{1,\ldots,m\}$ such that $(Ax)_q=0$. Let $a_q$ denote the $q$-th row of $A$, and let $\operatorname{Ker}(a_q)$
denote its null space. Then
\begin{equation*}
S_\infty=\cup_{q=1}^m\left(\operatorname{Ker}(a_q)\cap [0,a^n]\right).
\end{equation*}
Since the mapping $f_q(x)=\langle a_q,x\rangle$ is continuous and $\{0\}\subset\mathbb{R}$ is closed, each set
\begin{equation*}
\operatorname{Ker}(a_q)=f_q^{-1}(\{0\})
\end{equation*}
is closed in $\mathbb{R}^n$. Therefore, $S_\infty$, being a finite union and intersection of closed sets, is itself closed.
$\hfill\Box$
\end{proof}
\begin{remark}
An alternative proof of the second component can be obtained by extending the argument of \cite[Proposition 4.3]{daniele2026deep}. For this reason, in the previous result, we discussed the properties of the set $S_\infty$, which are necessary to prove \cite[Proposition 4.3]{daniele2026deep} but not Theorem \ref{theorem:conver}. It should be noted, however, that the objective functional $\Psi$ considered here is not globally subanalytic, but only locally subanalytic. Nevertheless, by \cite[Corollary 4.5]{daniele2026deep}, this local characterization is still sufficient to establish that $\Psi$ is a regularized K\L{} functional. In the present work, we instead adopt the $o$-minimal framework, since it covers a substantially broader class of functions and avoids the technical construction of auxiliary real-analytic functions required in \cite[SM 2.4]{daniele2026deep}. This provides a more direct and unified proof of the K\L{} property.
\end{remark}

\begin{theorem}[Convergence of mirror descent with AA-type fidelity]\label{theorem:conver} Let $\Psi=\mathcal{D}^{AA}(x;y)+\iota_{[0,a]^n}(x)+\lambda R_\theta(x)$ be the objective functional. Assume that the step size $\tau$ satisfies $0<\tau \bar{L}<1$, where $\bar{L}$ is defined in \eqref{eq:Lsmad_AA}. Then the sequence $\{x^k\}$ produced by the iteration \eqref{eq:MD_AA} converges to a critical point of $\Psi$.
\end{theorem}
\begin{proof}
The result follows by verifying the assumptions of \cite[Theorem 3]{hurault2023convergent} (see also \cite[Theorem 4.1]{bolte2018first}). Let $C_h=\operatorname{dom}(h^{AA})$. The required conditions are:

\begin{enumerate}[label=(\roman*)]
\item $h^{AA}:C_h\subset\mathbb{R}^n\to\mathbb{R}\cup\{+\infty\}$ is a proper, convex, $\mathcal{C}^2$ Legendre function, and is strongly convex on every bounded convex subset of $C_h$.

\item $F(x):=\mathcal{D}^{AA}(x;y)+\lambda R_\theta(x)$ is proper, continuously differentiable, and bounded from below on
$
\operatorname{conv}(\operatorname{dom}(\iota_{[0,a]^n}))\cap \operatorname{int}(\operatorname{dom}(h^{AA})).
$

\item The functional $\bar{L}h^{AA}-F$ is convex on
$
\operatorname{conv}(\operatorname{dom}(\iota_{[0,a]^n}))\cap \operatorname{int}(\operatorname{dom}(h^{AA})).
$
\item $\Psi=F+\iota_{[0,a]^n}$ is bounded from below, coercive, and satisfies the K\L{} property.

\item For every $\alpha>0$, both $\nabla h^{AA}$ and $\nabla F$ are Lipschitz continuous on the sublevel set
\[
\{x\in \operatorname{conv}(\operatorname{dom}(\iota_{[0,a]^n}))
\cap \operatorname{int}(\operatorname{dom}(h^{AA}))
:\Psi(x)\le \alpha\}.
\]

\item The set $[0,a]^n$ is nonempty, closed, convex, and semialgebraic.

\item For every $x^k\in \operatorname{int}(\operatorname{dom}(h^{AA}))$, the solution set
\[
M_\tau(x^k) \coloneqq
\arg\min_{x\in\mathbb{R}^n}
\Bigl\{
\iota_{[0,a]^n}(x)
+\langle x-x^k,\nabla F(x^k)\rangle
+\frac1\tau D_h^{AA}(x,x^k)
\Bigr\}
\]
is nonempty and contained in $\operatorname{int}(\operatorname{dom}(h^{AA}))$.
\end{enumerate}
Proposition \ref{pro:PsiKL} establishes condition (vi). In conjunction with the lower bound
\begin{equation*}
\mathcal{D}^{AA}(x;y)\geq m +\sum_{q=1}^m\log y_q,
\end{equation*}
it also yields conditions (ii) and (iv). Condition (iii) follows from Lemmas \ref{lema:1divxLhg} and \ref{lemma:rtheta_smooth}. Furthermore, by Definition \ref{def:legendre} and the explicit form of the chosen Bregman potential, $h^{AA}(x)=\sum_{j=1}^n\frac{1}{x_j}$ is a Legendre function and satisfies condition (i). It therefore remains to verify (v) and (vii).

For (v), let
\[
\Omega_\alpha \coloneqq
\{x\in \operatorname{conv}(\operatorname{dom}(\iota_{[0,a]^n}))
\cap \operatorname{int}(\operatorname{dom}(h^{AA}))
:\Psi(x)\le \alpha\}.
\]
Since $\operatorname{dom}\left(\iota_{[0,a]^n}\right)=[0,a]^n$ and $\operatorname{dom}(h^{AA})=(0,+\infty)^n$, we have
\[
\Omega_\alpha\subset (0,a]^n.
\]
Hence $\Omega_\alpha$ is bounded. Because $\Psi(x)\leq\alpha$ on $\Omega_\alpha$, there exists some $\beta>0$, such that $\Omega_\alpha\subset [\beta,a]^n$. Since both $h^{AA}$ and $F$ are $\mathcal{C}^2$ on $[\beta,a]^n$, their Hessians are bounded on $\Omega_\alpha$. Consequently, $\nabla h^{AA}$ and $\nabla F$ are Lipschitz continuous on $\Omega_\alpha$, establishing (v).

To verify (vii), note that $\iota_{[0,a]^n}$ is convex and $h^{AA}$ is strictly convex on $\operatorname{int}(\operatorname{dom}(h^{AA}))$. Therefore, the surrogate objective
\[
x\mapsto
\iota_{[0,a]^n}(x)
+\langle x-x^k,\nabla F(x^k)\rangle
+\frac1\tau D_h^{AA}(x,x^k)
\]
is strictly convex on $[0,a]^n$. Since $\iota_{[0,a]^n}$ restricts the feasible set to the compact set $[0,a]^n$, the above minimization problem admits a unique global minimizer.

Furthermore, for any $x^k\in \operatorname{int}(\operatorname{dom}(h^{AA}))$, the Bregman distance satisfies
\[
D_h^{AA}(x,x^k)=+\infty
\]
whenever $x\in \partial\operatorname{dom}(h^{AA})$. Hence the minimizer cannot lie on the boundary of $\operatorname{dom}(h^{AA})$, implying that
\[
M_\tau(x^k)\subset \operatorname{int}(\operatorname{dom}(h^{AA})).
\]
Therefore $M_\tau$ is a well-defined single-valued mapping from $\operatorname{int}(\operatorname{dom}(h^{AA}))$ into itself, consistent with \cite[Remark 3.1]{bolte2018first}. Similarly, one can see from the iterative update \eqref{eq:MD_AA} that when $x^0\in \operatorname{int}(\operatorname{dom}(h^{AA}))$, the pre-projection result still lies in $\mathbb{R}^n_{++}$ rather than in $\partial\operatorname{dom}(h^{AA})$.

All assumptions of \cite[Theorem 3]{hurault2023convergent} are thus satisfied, and the conclusion follows immediately.
$\hfill\Box$
\end{proof}

\begin{proposition}[Existence and convergence of fixed points] Define 
\begin{equation}
f_\theta(x;y)=\Pi_{[0,a]^n}\left(\frac{x}{\sqrt{1+\tau x^2\left(\nabla R_\theta(x)+\nabla \mathcal{D}^{AA}(x;y)\right)}}\right).
\label{eq:f_theta}
\end{equation}
Then the AA-type update rule \eqref{eq:MD_AA} can be written as $x^{k+1}=f_\theta(x^k;y)$. Suppose that $0<\tau\bar{L}<1$, where $\bar{L}$ is defined in \eqref{eq:Lsmad_AA}, and let
$x^0\in\operatorname{int}(\operatorname{dom}(h^{AA}))$.
Then the fixed-point set
\begin{equation*}
\operatorname{Fix}(f_\theta(\cdot;y))=\{x\in\mathbb{R}^n:x=f_\theta(x;y)\}
\end{equation*}
is nonempty. Moreover, the sequence $\{x^k\}$ generated by \eqref{eq:MD_AA} converges to a point in $\operatorname{Fix}(f_\theta(\cdot;y))$.
\end{proposition}
\begin{proof}
By Theorem \ref{theorem:conver}, the sequence $\{x^k\}$ converges to some limit point $x^\infty\in\operatorname{int}(\operatorname{dom}(h^{AA}))$. Since $R_\theta$ and $\mathcal{D}^{AA}$ are analytic, the mapping $f_\theta$ defined in
\eqref{eq:f_theta} is continuous. Passing to the limit in
$x^{k+1}=f_\theta(x^k;y)$ gives
\begin{equation*}
x^\infty=\lim_{k\rightarrow \infty}x^{k+1}=\lim_{k\rightarrow \infty}f_\theta(x^k;y)=f_\theta(\lim_{k\rightarrow \infty} x^k;y)= f_\theta(x^\infty;y).
\end{equation*}
Hence $x^\infty\in\operatorname{Fix}(f_\theta(\cdot;y))$ showing that the fixed-point set is nonempty. Since $\{x^k\}$ converges to $x^\infty$, the proof is complete.
$\hfill\Box$
\end{proof}
\section{Forward and gradient computation in DEQs}\label{sec:ForBack}

Having established the theoretical framework of the proposed DEQ model at inference stage, two computational challenges remain to be addressed in practice. First, during the forward pass, one must efficiently compute a fixed point of the mapping $f_\theta$ for a given observation and network parameters $\theta\in\Theta$. Second, since a DEQ model can be interpreted as an infinitely deep network, standard backpropagation is generally infeasible due to its prohibitive memory requirements. Consequently, after solving the lower-level fixed-point problem in \eqref{eq:bilevelproblm} and evaluating the upper-level loss, an efficient and memory-intensive strategy is required to optimize the network parameters $\theta$.

In the remainder of this section, we describe the numerical techniques employed to address these two challenges.
\subsection{Fixed-point computation via backtracking mirror descent}
Within the DEQ framework, both training and inference require the computation of a fixed point of the mapping $f_\theta$. A straightforward approach is to apply the fixed-point iteration induced by \eqref{eq:MD_KL} and \eqref{eq:MD_AA}. The efficiency and convergence behavior of this iteration, however, depend critically on the choice of the step size $\tau$.

The convergence of the AA iteration \eqref{eq:MD_AA} is guaranteed by Theorem \ref{theorem:conver}, provided that $\tau<\frac{1}{\bar{L}}$, where $\bar{L}=L+L^\prime$ with the constants $L$ and $L^\prime$ given in Lemmas \ref{lema:1divxLhg} and \ref{lemma:rtheta_smooth}, respectively. For the KL iteration \eqref{eq:MD_KL}, the same step-size condition is required in \cite[Theorem 3]{hurault2023convergent}, where $L$ is characterized in Lemma \ref{lema:BurgLhg} and $L^\prime$ is derived in \cite[SM 2.3.3]{daniele2026deep}. Nevertheless, such a strategy is impractical for several reasons.

First, as discussed in Lemma \ref{lemma:rtheta_smooth}, estimating the $L$-smad constant $L^\prime$ of the network requires evaluating bounds on the spectral radius of the Hessian of $R_\theta$, which can be computationally expensive. Second, the resulting estimates are often overly conservative \cite{hurault2023convergent}, leading to unnecessarily small step sizes and consequently slow convergence. Finally, the validity of the update itself imposes additional restrictions on $\tau$.   
Indeed, inspection of the update formula \eqref{eq:MD_AA} shows that the denominator terms must remain strictly positive; otherwise, the square-root operation becomes undefined. 
Therefore, an appropriate step-size selection mechanism is essential not only for convergence but also for the well-definedness of the iteration.

\begin{remark} (Step size for the KL fidelity). Although the update formula \eqref{eq:MD_KL} associated with the KL-type fidelity term does not explicitly involve a square-root operation, positivity constraints remain necessary. Indeed, for Burg's entropy $h(x)=-\log x$, the Fenchel conjugate is given by $h^*(y)=-\log(-y)-1$, whose domain is $(-\infty,0)$ \cite{bauschke2017descent}. Comparing \eqref{eq:MD} with the update formula \eqref{eq:MD_KL} shows that the corresponding denominator terms must still remain positive in order for the iteration to be well defined.
\end{remark}

Motivated by the backtracking principle introduced in \cite{beck2017first}, we develop an adaptive step-size strategy specifically for the proposed mirror descent scheme. Starting from an initial step size $\tau_0>0$, the step size is decreased adaptively until the sufficient decrease condition \eqref{eq:backtrack_energy} and the positivity condition \eqref{eq:backtrack_positive} are both satisfied. The former is formulated in terms of the Bregman distance $D_h$, while the latter guarantees the well-definedness of the subsequent iterate. Consequently, the explicit estimation of the $L$-smad constants is avoided without sacrificing the convergence guarantees of the mirror descent scheme. The complete procedure is summarized in Algorithm \ref{alg:DEQ_backtrack}.

\begin{algorithm}[t]
\caption{Adaptive Bregman Backtracking Step-Size Selection}
\label{alg:DEQ_backtrack}
\begin{algorithmic}[1]

\Require
Current iterate $x^k\in\operatorname{int}(\operatorname{dom}h)$, initial step size $\tau_0>0$,
sufficient decrease parameter $\gamma\in(0,1)$,
backtracking factor $\eta\in(0,1)$.

\Ensure
Updated iterate $x^{k+1}$.

\State Set $\tau \leftarrow \tau_0$.

\State Compute
\[
M_\tau(x^k)=
\operatorname{Prox}^{h}_{\tau\iota_{[0,a]^n}}
\!\left(
\nabla h^{*}
\!\left(
\nabla h(x^k)
-\tau
\bigl(
\nabla R_{\theta}(x^k)
+\nabla\mathcal{D}(x^k;y)
\bigr)
\right)
\right).
\]

\While{
\begin{equation}
\Psi(x^k)-\Psi(M_\tau(x^k))
\leq
\frac{\gamma}{\tau}
D_h(M_\tau(x^k),x^k)
\label{eq:backtrack_energy}
\end{equation}
or
\begin{equation}
\nabla h(x^k)
-\tau
\bigl(
\nabla R_\theta(x^k)
+\nabla\mathcal{D}(x^k;y)
\bigr)
\geq 0
\label{eq:backtrack_positive}
\end{equation}
}

\State $\tau \leftarrow \eta\tau$.

\State Recompute $M_\tau(x^k)$.

\EndWhile

\State Set
\[
x^{k+1}\leftarrow M_\tau(x^k).
\]

\State \Return $x^{k+1}$.

\end{algorithmic}
\end{algorithm}
Having described the algorithm, we now present the following proposition regarding its convergence and the finite termination of the backtracking procedure.

\begin{theorem}
The backtracking procedure in Algorithm \ref{alg:DEQ_backtrack} terminates after finitely many iterations. Moreover, the sequence $\{x^k\}$ generated by \eqref{eq:MD} with step sizes determined by the backtracking strategy converges to a critical point of $\Psi$.
\end{theorem}
\begin{proof}
The backtracking strategy preserves the sufficient decrease property required in the proof of Theorem \ref{theorem:conver}. Therefore, the convergence analysis extends to the backtracking scheme, and the convergence of the generated sequence follows by the same argument; see \cite[Proposition 4]{hurault2023convergent}. It therefore remains to show that the backtracking procedure terminates in finitely many steps.

By \cite[Lemma 2]{hurault2023convergent}, for any step size $\tau>0$, the following sufficient decrease estimate holds:
\begin{equation*}
\Psi(x^k)-\Psi(M_\tau(x^k))\geq \left(\frac{1}{\tau}-\bar{L}\right)D_h(M_\tau(x^k),x^k),
\end{equation*}
where $\bar{L}>0$ denotes the Lipschitz constant of $\lambda R_\theta+\mathcal{D}$. Consequently, whenever $\tau<\frac{1-\gamma}{\overline{L}}$, we have $\frac{1}{\tau}-\bar{L}> \frac{\gamma}{\tau}$, which implies
\begin{equation*}
\Psi(x^k)-\Psi(M_\tau(x^k))> \frac{\gamma}{\tau}D_h(M_\tau(x^k),x^k).
\end{equation*}
Hence the first acceptance condition in the while-loop of Algorithm \ref{alg:DEQ_backtrack} is automatically satisfied for $\tau$ chosen in this way.

Next, define
\begin{equation*}
T_h(\tau)=\nabla h(x^k)-\tau\bigl(\nabla R_\theta(x^k)+\nabla\mathcal{D}(x^k;y)\bigr).
\end{equation*}
At $\tau=0$, one has
\begin{equation*}
T_h(0) =
\begin{cases} 
-\frac{\mathbf{1}}{x^k}, & \text{for the KL-type fidelity,} \\
-\frac{\mathbf{1}}{(x^k)^2}, & \text{for the AA-type fidelity,}
\end{cases}
\end{equation*}
where $\mathbf{1}\in\mathbb{R}^n$ denotes the vector of ones.
When $x^k\in (0,a]^n$, $T_h(0)$ is strictly negative componentwise in both cases. By continuity of $T_h$, there exists $\delta_\mathcal{D}>0$ such that $T_h(\tau)<0$ for all $0<\tau<\delta_\mathcal{D}$. 

Combining the two observations, all acceptance conditions are fulfilled for $\tau<\min\{\delta_\mathcal{D}, \frac{1-\gamma}{\bar{L}}\}$. Since the backtracking strategy decreases the step size geometrically, it must terminate after finitely many reductions. This completes the proof.
$\hfill\Box$
\end{proof}

Finally, for the implementation of the forward pass, it is important to specify how to evaluate the gradient $\nabla R_{\theta}$ appearing in \eqref{eq:MD}. To do so, two strategies are available. The most direct approach is to employ automatic differentiation to evaluate $\nabla R_{\theta}$ and substitute it into the fixed-point mapping $f_\theta$. Alternatively, following the idea of \cite{chen2016trainable}, one may exploit the explicit form of the regularizer in \eqref{eq:regular1} to derive $\nabla R_{\theta}$ analytically, which can then be incorporated into the forward iteration \eqref{eq:MD}. Although an analytical implementation may offer additional computational efficiency, it is generally less flexible and requires problem-dependent derivations. In particular, explicitly differentiating the terms involved in $R_\theta$, such as the surface area $\mathcal{A}$ and curvature $\mathcal{K}$, often results in cumbersome expressions. Moreover, certain components of the functional may vary depending on the specific application, including the normalization strategies adopted in $\psi_l^a$ and the treatment of convolutions across multiple image channels. By contrast, automatic differentiation requires no model-specific derivation, is straightforward to implement, and benefits from the optimized differentiation routines provided by modern deep learning frameworks such as PyTorch. For these reasons, we adopt the automatic differentiation implementation throughout this work.

\subsection{Gradient computation via implicit differentiation}
After computing the lower-level solution of the bilevel problem \eqref{eq:bilevelproblm} via mirror descent, we evaluate the upper-level objective function
\[
\mathcal{L}(\theta)=\frac{1}{N}\sum_{i=1}^N\ell\left(
x_i^\infty\big(\theta;(x_i^0,y_i)\big),x_i^*\right).
\]
The network parameters $\theta$ are then updated using standard optimization algorithms, such as first-order stochastic methods (e.g., Adam \cite{kingma2014adam}) or quasi-Newton schemes (e.g., L-BFGS \cite{liu1989limited}). Consequently, an efficient procedure for computing gradients with respect to $\theta$ is essential in the DEQ framework, whose infinite depth precludes the direct application of conventional backpropagation.

A well-known method to overcome this issue is to rely on implicit differentiation (see \cite{gilton2021deep}).

For clarity of exposition, we consider a single pair consisting of an observation and its ground truth, denoted by $(y,x^*)$. Denote the loss by
\[
\ell_\theta:=\ell\left(
x^\infty\big(\theta;(x^0,y)\big),x^*\right)
\]
and write $x^\infty\big(\theta;(x^0,y_i)\big)$ as shorthand for $x^\infty$. Since $x^\infty$ satisfies the fixed-point relation $x^\infty=f_\theta\left(x^\infty;y\right)$, we can treat $x^\infty$ as an implicit function of $\theta$. By the chain rule, the gradient of the loss with respect to the network parameters is given by
\begin{equation}
\frac{\partial \ell_\theta}{\partial \theta}
=
\left(\frac{\partial x^\infty}{\partial \theta}\right)^T
\frac{\partial \ell_\theta}{\partial x^\infty}.
\label{eq:chain}
\end{equation}
 For commonly used losses such as the mean squared error, the term $\frac{\partial \ell_\theta}{\partial x^\infty}$ admits a simple closed-form expression. The main computational challenge therefore lies in evaluating $\frac{\partial x^\infty}{\partial \theta}$.

This difficulty can be substantially simplified by exploiting the fixed-point relation satisfied by the equilibrium state. Since $x^\infty$ is a fixed point of $f_\theta(\cdot;y)$, it satisfies
\[
x^\infty = f_\theta(x^\infty;y).
\]
Differentiating both sides with respect to $\theta$ and rearranging the resulting expression yields
\[
\frac{\partial x^\infty}{\partial \theta}
=
\left(I - \left. \frac{\partial f_\theta}{\partial x} \right|_{x = x^\infty} \right)^{-1}
\frac{\partial f_\theta(x^\infty;y)}{\partial \theta},
\]
where $\frac{\partial f_\theta(x^\infty;y)}{\partial \theta}$ denotes the Jacobian of $f_\theta$ with respect to $\theta$, evaluated at the equilibrium point $x^\infty$.
Substituting this identity into the chain rule \eqref{eq:chain} leads to the final gradient formula
\[
\frac{\partial \ell_\theta}{\partial \theta}
=
\left(\frac{\partial f_\theta(x^\infty;y)}{\partial \theta}\right)^T
\left(I - \left. \frac{\partial f_\theta}{\partial x} \right|_{x = x^\infty} \right)^{-T}
\frac{\partial \ell_\theta}{\partial x^\infty}.
\]


The Jacobian $\frac{\partial f_\theta(x^\infty;y)}{\partial \theta}$ can be computed via automatic differentiation, once the fixed point $x^\infty$ is obtained using Algorithm \ref{alg:DEQ_backtrack} for a fixed $\theta$. 
Consequently, the only remaining task is the evaluation of
$\left(I - \left. \frac{\partial f_\theta}{\partial x} \right|_{x = x^\infty} \right)^{-T}$. Rather than solving the associated linear system explicitly, as in \cite{bai2019deep,gilton2021deep}, we adopt a Jacobian-free backpropagation (JFB) strategy \cite{fung2022jfb}. This approach avoids the computational overhead of large-scale linear system solves, requires only fixed memory, and has been observed to provide competitive gradient accuracy in practice.

At an implementation level, JFB approximates the inverse-transpose Jacobian term $\left(I - \left. \frac{\partial f_\theta}{\partial x} \right|_{x = x^\infty} \right)^{-T}$ by the identity operator $I$. Although this approximation may appear heuristic, it admits several rigorous interpretations. One viewpoint is to regard the equilibrium state $f_\theta(x^\infty;y)$ as independent of the network input $x^\infty$ during differentiation, effectively treating the fixed point as if it were obtained after a single converged iteration. For this reason, JFB is often referred to as one-step differentiation \cite{bolte2023one}. An attractive consequence of this interpretation is that the resulting gradient can be computed directly within standard automatic differentiation frameworks. From an optimization perspective, JFB can also be interpreted as a preconditioned gradient method in which the exact gradient is replaced by an alternative descent direction obtained through a suitable preconditioning strategy; see \cite{fung2022jfb} for a detailed analysis. Another interpretation views JFB as the zeroth-order truncation of the Neumann-series expansion of the inverse Jacobian term, as discussed in \cite{daniele2026deep}.

These theoretical justifications, together with its favorable computational efficiency, have made JFB a widely adopted approximation strategy in the training of equilibrium models \cite{zou2023deep,gan2023self}.

\section{Experimental results}\label{sec:Exper}
To quantitatively assess reconstruction quality, we first introduce three widely used quantitative metrics for evaluating reconstruction quality, each characterizing a different aspect of image fidelity.
\begin{itemize}
\item[$\bullet$] PSNR (Peak Signal‑to‑Noise Ratio) \cite{jahne2005digital}: Measures the pixel‑wise mean squared error between the reconstructed image and the ground truth, expressed in decibels, reflecting fidelity at the pixel level.
\item[$\bullet$] SSIM (Structural Similarity Index) \cite{wang2004image}: Computes a weighted combination of luminance, contrast, and structural components, yielding a value in $[0,1]$ that quantifies structural similarity.
\item[$\bullet$] LPIPS (Learned Perceptual Image Patch Similarity) \cite{zhang2018unreasonable}: Based on a pretrained deep neural network (AlexNet in our experiments), computes a distance in feature space between the reconstruction and the ground truth. This metric, typically nonnegative, correlates well with human perceptual judgment.
\end{itemize}
\subsection{Training details}
The proposed model is trained on the DIV2K dataset \cite{agustsson2017ntire}, which contains a diverse collection of high-resolution natural images. 
The 720 training images and 80 validation images are cropped to a size of $(256 \times 256)$ before training.
All experiments are conducted on a workstation equipped with an Intel Core i9-12900K CPU, 125 GB RAM, and an NVIDIA RTX A5000 GPU with 24 GB memory, running Ubuntu 24.04.4 LTS.

\subsection{Parameter and function initialization}
The area-related influence function is initialized by the smooth approximation $\psi^a(z)=\log(1+z^2)$, while the curvature-related influence function is initialized as the quadratic penalty $\psi^c(z)=z^2$. The corresponding RBF coefficients are obtained by least-squares fitting, where the inputs are normalized and rescaled to ensure that they lie within the effective support of the basis functions. The implementation details are available in the accompanying code repository \url{https://github.com/Shengkunyang/DEQ4Gamma}. To ensure positivity, we parameterize $\epsilon\leftarrow \exp(\tilde{\epsilon})$ and $b=\exp(\tilde{b})$, with $\tilde{\epsilon}$ and $\tilde{b}$ initialized such that $\epsilon=10^{-6}$ and $b=0.01$. Since the learned influence functions are expected to adapt the relative scaling between the data fidelity and regularization terms, the balancing parameter $\lambda$ is fixed to 1 throughout all experiments. Given a degraded observation $y$, the initial estimate of the DEQ model is chosen as $x^0=A^\top y$. Each convolution kernel is initialized as a random linear combination of DCT basis filters. During training, the parameters $\theta$ are learned by minimizing the mean squared error loss $\mathcal{L}(\theta)$.
For grayscale image restoration, the network employs $12$ learnable $5\times 5$ convolution kernels, each parameterized by $24$ DCT basis filters ($N_r=24$), together with $33$ Gaussian radial basis functions ($M=33$). The regularizer contains a total of 13 learnable influence functions ($N_a+N_c=13$). For color image restoration, to reduce GPU memory consumption, the numbers of area- and curvature-related influence functions are set to $N_a=8$ and $N_c=3$, respectively, while all other network configurations remain unchanged. The resulting parameterized architecture is illustrated in Fig. \ref{fig:Overall_architecture}.
\begin{figure}[hbtp]
    \centering
    \begin{minipage}{\linewidth}
        \centering
        \includegraphics[width=\linewidth]{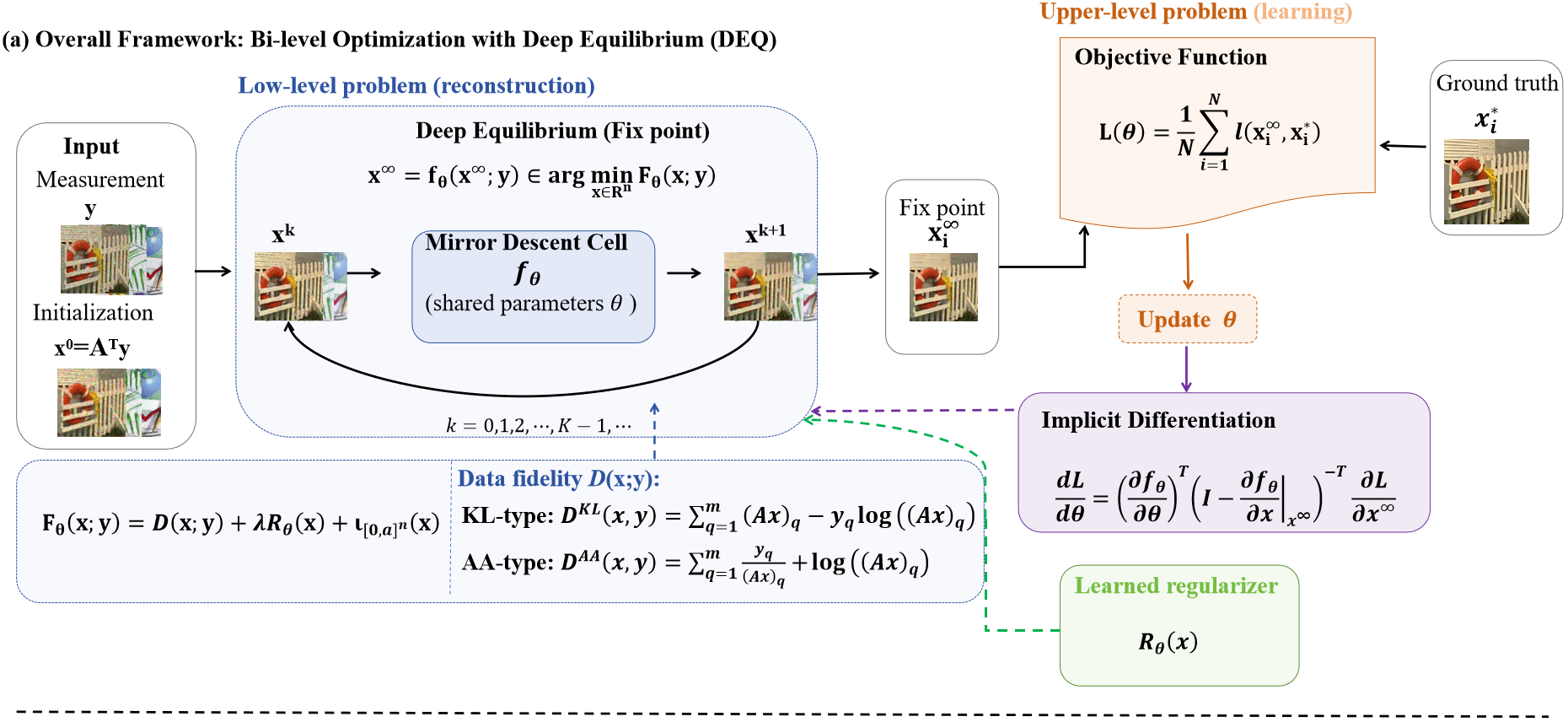}
    \end{minipage}

    \vspace{0.3em}
    \begin{minipage}{0.495\linewidth}
        \centering
        \includegraphics[width=\linewidth]{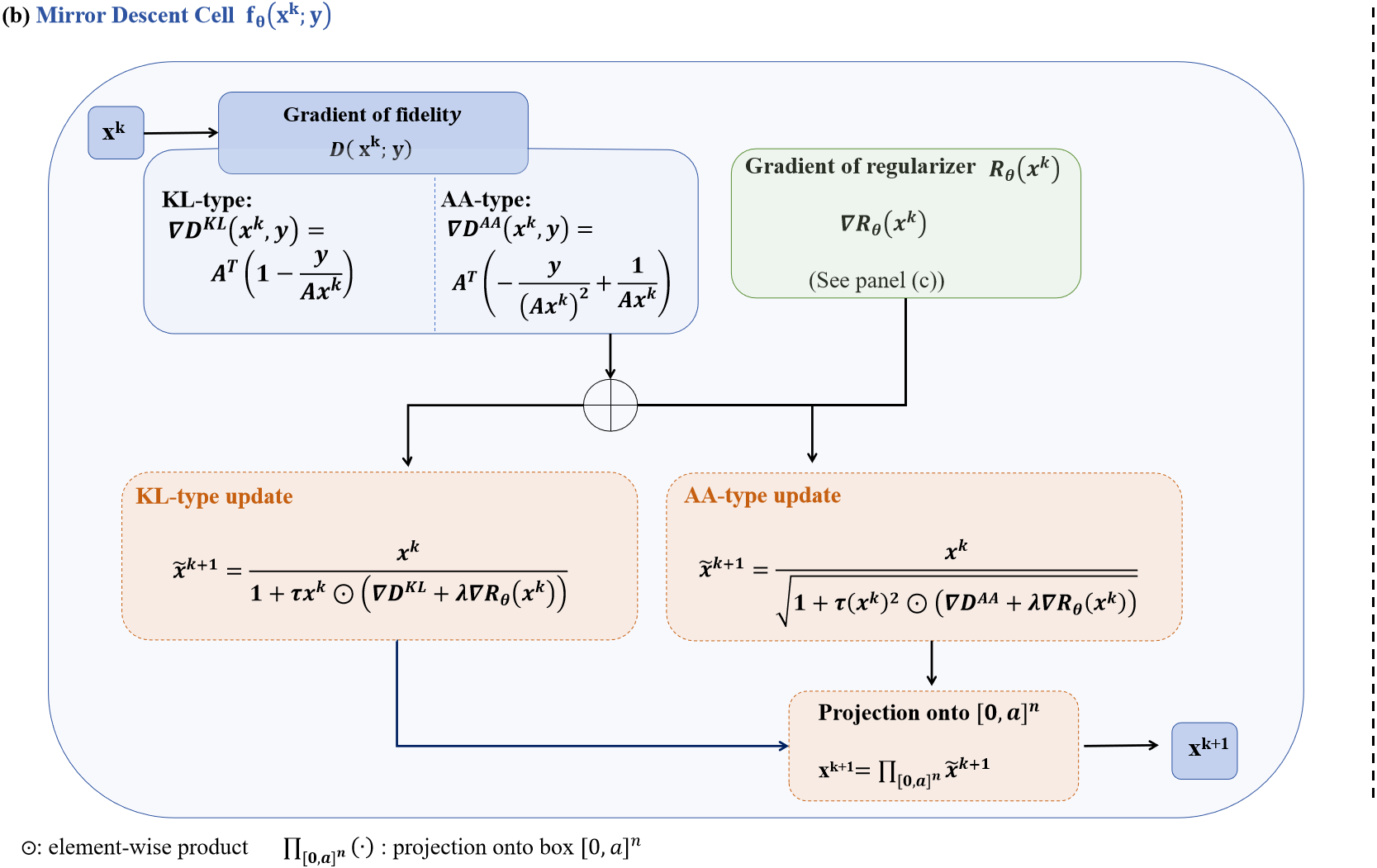}
    \end{minipage}
    \hfill
    \begin{minipage}{0.495\linewidth}
        \centering
        \includegraphics[width=\linewidth,height=0.64\linewidth]{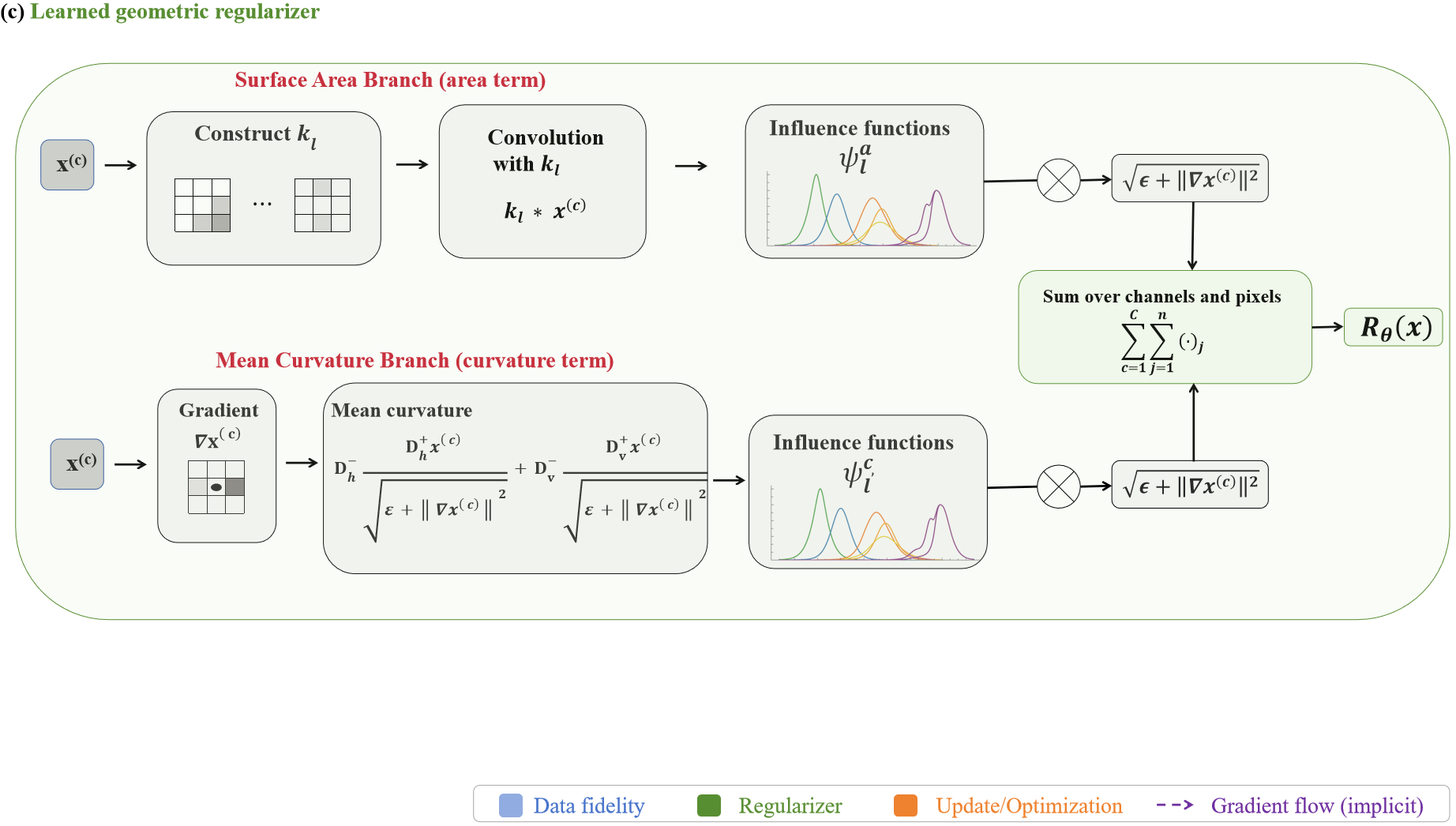}
    \end{minipage}
    \setlength{\abovecaptionskip}{10pt}
    \caption{Overall architecture of the proposed Mirror Descent-based Deep Equilibrium (DEQ) model. (a) Overall pipeline of the proposed framework. (b) Mirror Descent module for equilibrium computation. (c) Geometry-aware learning architecture for estimating the regularization operator.}
    \label{fig:Overall_architecture}
\end{figure}
\subsection{Gray image restoration}
We evaluate the proposed method on grayscale image restoration under two blur settings: a Gaussian blur with standard deviation $2$ and a motion blur with length $5$ and angle $\ang{30}$. After blurring, the images are contaminated with multiplicative Gamma noise with equivalent numbers of looks $L=4$ and $L=10$. To the best of our knowledge, although a deep learning method has recently been proposed for joint deblurring and multiplicative Gamma denoising \cite{li2023deep}, no publicly available implementation is currently available, preventing a direct experimental comparison. Fortunately, the literature provides a number of well-established model-based methods for this challenging restoration task. Accordingly, we compare the proposed method with five representative model-based approaches, namely the AA \cite{aubert2008variational}, RLO \cite{rudin2003multiplicative}, DZ \cite{dong2013convex}, ZWN \cite{zhao2014new}, and MG \cite{yang2025mixed} models. 
For each approach, the regularization parameters are tuned independently for each image to achieve the best reconstruction performance.
For the proposed method, the variants employing the KL-type and AA-type data fidelity terms are denoted by DEQ-KL and DEQ-AA, respectively. The same notation is adopted throughout the color image restoration experiments. 

The experiments are performed on three representative grayscale images: the synthetic image ``Dart'', the SAR image ``SAR'', and the natural image ``Leaves'', which are shown in Figure \ref{fig:gray_allgt}.

\begin{figure}[htbp]
    \centering
    \begin{minipage}{\textwidth}
        \centering
        \subfigcapskip=1pt
        \subfigure[Dart]{
            \includegraphics[width=0.22\textwidth]{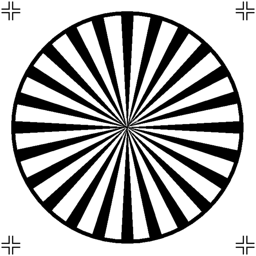}
        }
        \subfigure[SAR]{
            \includegraphics[width=0.22\textwidth]{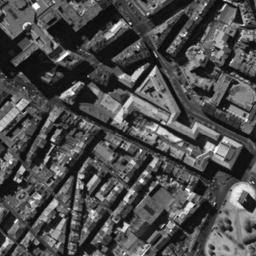}
        }
        \subfigure[Leaves]{
            \includegraphics[width=0.22\textwidth]{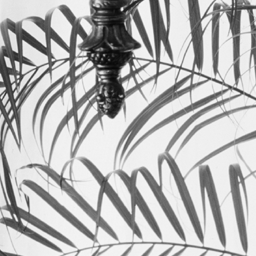}
        }
    \end{minipage}
  \caption{Three test images for grayscale image restoration experiments.}
    \label{fig:gray_allgt}
\end{figure}

\begin{figure}[htbp]
    \centering
    \begin{minipage}{\textwidth}
        \centering
        \subfigcapskip=1pt
        \subfigure[Blurred \& Noisy]{
            \includegraphics[width=0.22\textwidth]{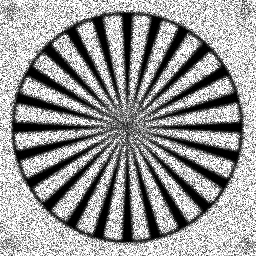}
        }
        \subfigure[AA]{
            \includegraphics[width=0.22\textwidth]{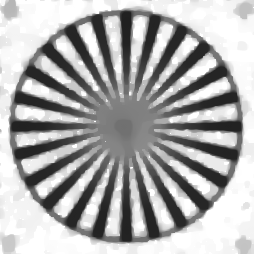}
        }
        \subfigure[RLO]{
            \includegraphics[width=0.22\textwidth]{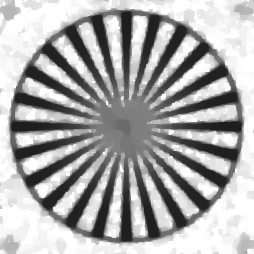}
        }
        \subfigure[DZ]{
            \includegraphics[width=0.22\textwidth]{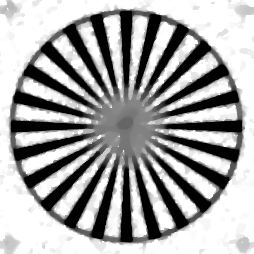}
        }
    \end{minipage}

    \vspace{0.5em}
    \begin{minipage}{\textwidth}
        \centering
        \subfigcapskip=1pt
        \subfigure[ZWN]{
            \includegraphics[width=0.22\textwidth]{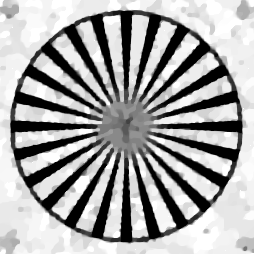}
        }
        \subfigure[MG]{
            \includegraphics[width=0.22\textwidth]{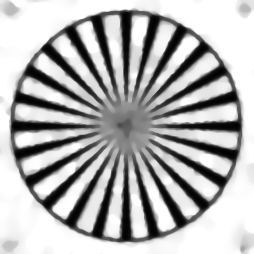}
        }
        \subfigure[DEQ-KL]{
            \includegraphics[width=0.22\textwidth]{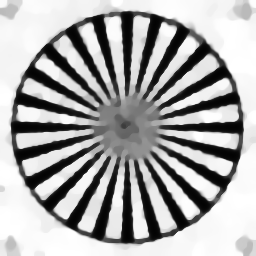}
        }
        \subfigure[DEQ-AA]{
            \includegraphics[width=0.22\textwidth]{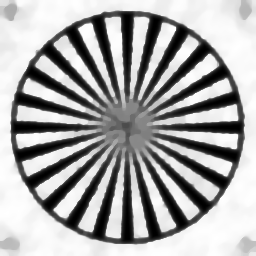}
        }
    \end{minipage}
  \caption{Comparison of restoration results obtained by different methods on the Dart image degraded by Gaussian blur and noise with noise level $L=4$.}
    \label{fig:comparison_gray_Gaussian_meidum_Dart}
\end{figure}

\begin{figure}[htbp]
    \centering
    \begin{minipage}{\textwidth}
        \centering
        \subfigcapskip=1pt
        \subfigure[Blurred \& Noisy]{
            \includegraphics[width=0.22\textwidth]{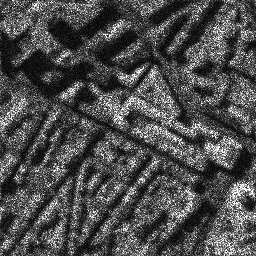}
        }
        \subfigure[AA]{
            \includegraphics[width=0.22\textwidth]{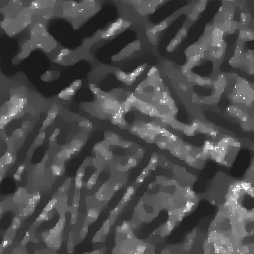}
        }
        \subfigure[RLO]{
            \includegraphics[width=0.22\textwidth]{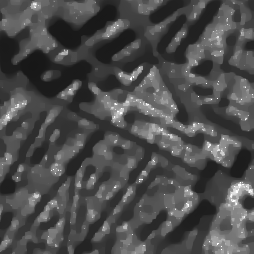}
        }
        \subfigure[DZ]{
            \includegraphics[width=0.22\textwidth]{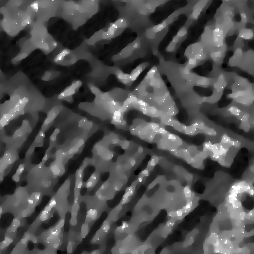}
        }
    \end{minipage}

    \vspace{0.5em}
    \begin{minipage}{\textwidth}
        \centering
        \subfigcapskip=1pt
        \subfigure[ZWN]{
            \includegraphics[width=0.22\textwidth]{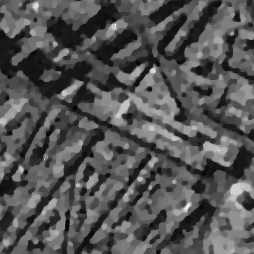}
        }
        \subfigure[MG]{
            \includegraphics[width=0.22\textwidth]{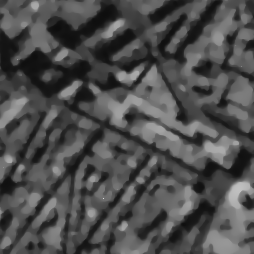}
        }
        \subfigure[DEQ-KL]{
            \includegraphics[width=0.22\textwidth]{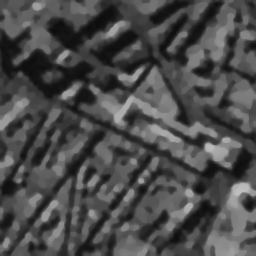}
        }
        \subfigure[DEQ-AA]{
            \includegraphics[width=0.22\textwidth]{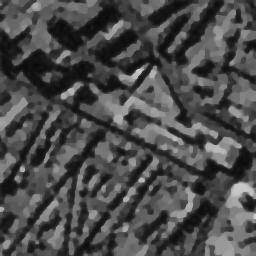}
        }
    \end{minipage}
  \caption{Comparison of restoration results obtained by different methods on the SAR image degraded by Gaussian blur and noise with noise level $L=4$.}
    \label{fig:comparison_gray_Gaussian_meidum_SAR}
\end{figure}
\begin{figure}[htbp]
    \centering
    \begin{minipage}{\textwidth}
        \centering
        \subfigcapskip=1pt
        \subfigure[Blurred \& Noisy]{
            \includegraphics[width=0.22\textwidth]{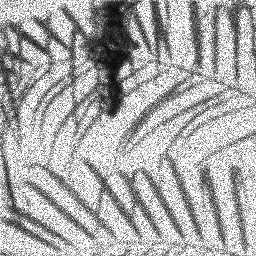}
        }
        \subfigure[AA]{
            \includegraphics[width=0.22\textwidth]{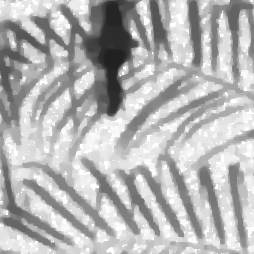}
        }
        \subfigure[RLO]{
            \includegraphics[width=0.22\textwidth]{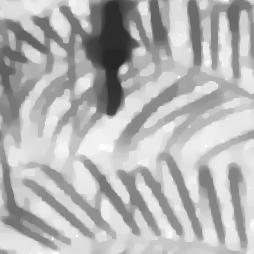}
        }
        \subfigure[DZ]{
            \includegraphics[width=0.22\textwidth]{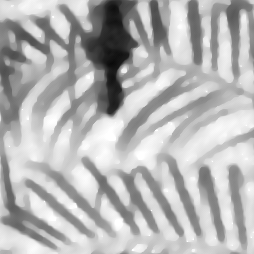}
        }
    \end{minipage}

    \vspace{0.5em}
    \begin{minipage}{\textwidth}
        \centering
        \subfigcapskip=1pt
        \subfigure[ZWN]{
            \includegraphics[width=0.22\textwidth]{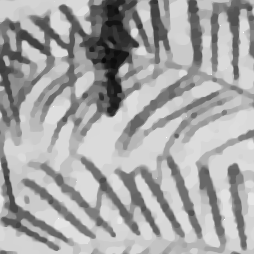}
        }
        \subfigure[MG]{
            \includegraphics[width=0.22\textwidth]{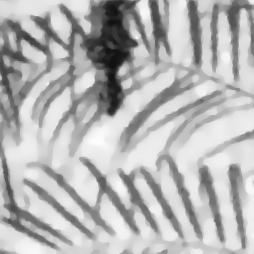}
        }
        \subfigure[DEQ-KL]{
            \includegraphics[width=0.22\textwidth]{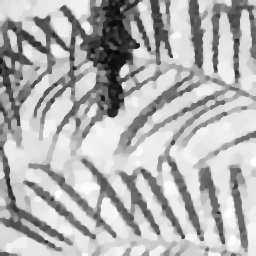}
        }
        \subfigure[DEQ-AA]{
            \includegraphics[width=0.22\textwidth]{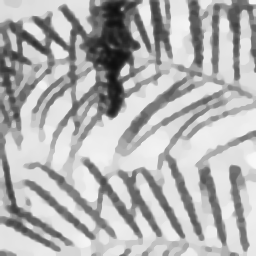}
        }
    \end{minipage}
  \caption{Comparison of restoration results obtained by different methods on the Leaves image degraded by Gaussian blur and noise with noise level $L=10$.}
    \label{fig:comparison_gray_Gaussian_low_Leaves}
\end{figure}

\begin{figure}[htbp]
    \centering
    \begin{minipage}{\textwidth}
        \centering
        \subfigcapskip=1pt
        \subfigure[Blurred \& Noisy]{
            \includegraphics[width=0.22\textwidth]{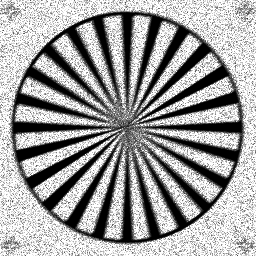}
        }
        \subfigure[AA]{
            \includegraphics[width=0.22\textwidth]{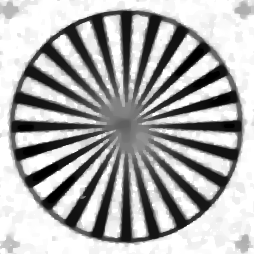}
        }
        \subfigure[RLO]{
            \includegraphics[width=0.22\textwidth]{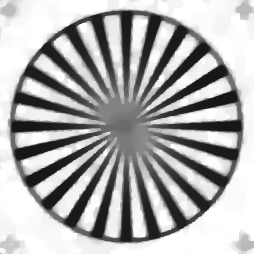}
        }
        \subfigure[DZ]{
            \includegraphics[width=0.22\textwidth]{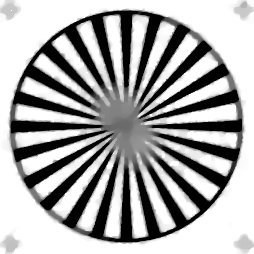}
        }
    \end{minipage}

    \vspace{0.5em}
    \begin{minipage}{\textwidth}
        \centering
        \subfigcapskip=1pt
        \subfigure[ZWN]{
            \includegraphics[width=0.22\textwidth]{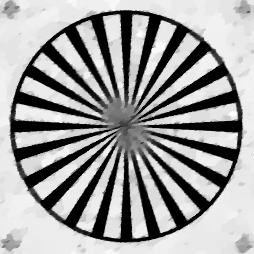}
        }
        \subfigure[MG]{
            \includegraphics[width=0.22\textwidth]{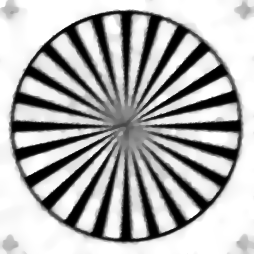}
        }
        \subfigure[DEQ-KL]{
            \includegraphics[width=0.22\textwidth]{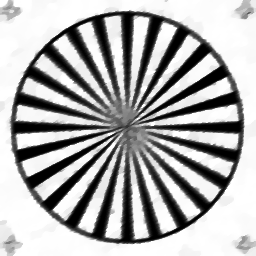}
        }
        \subfigure[DEQ-AA]{
            \includegraphics[width=0.22\textwidth]{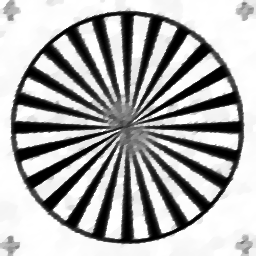}
        }
    \end{minipage}
  \caption{Comparison of restoration results obtained by different methods on the Dart image degraded by motion blur and noise with noise level $L=10$.}
    \label{fig:comparison_gray_Motion_low_Dart}
\end{figure}

\begin{table}[htbp]
\centering
\caption{Comparison of PSNR, SSIM, and LPIPS achieved by different restoration methods on grayscale test images degraded by Gaussian blur and noise under various noise levels.}
\label{tab:comparison_gray_Gaussian_metrics}
\begin{tabular}{ccclccc}
\toprule
Image & $L$ & Method & PSNR$\uparrow$ & SSIM$\uparrow$ & LPIPS$\downarrow$ \\
\midrule

\multirow{14.5}{*}{Dart}
& \multirow{7}{*}{4}
& AA & 13.52 & 0.49 & 0.44 \\
&& RLO & 13.56 & 0.44 & 0.51 \\
&& DZ & 14.49 & 0.50 & 0.52 \\
&& ZWN & \textbf{17.09} & 0.78 & 0.16 \\
&& MG & 15.47 & 0.70 & 0.21 \\
& & DEQ-KL & 16.05 & \textbf{0.79} & 0.15 \\
& & DEQ-AA & 16.07 & 0.76 & \textbf{0.14} \\
\cmidrule(lr){2-6}

& \multirow{7}{*}{10}
& AA & 14.97 & 0.51 & 0.41 \\
&& RLO & 14.97 & 0.52 & 0.39 \\
&& DZ & 15.15 & 0.53 & 0.47 \\
&& ZWN & 17.34 & 0.88 & 0.11 \\
&& MG & 16.19 & 0.73 & 0.21 \\
& & DEQ-KL & \textbf{18.95} & \textbf{0.86} & \textbf{0.09} \\
& & DEQ-AA & 18.80 & 0.85 & \textbf{0.09} \\
\midrule

\multirow{14.5}{*}{SAR}
& \multirow{7}{*}{4}
& AA & 19.41 & 0.46 & 0.36 \\
&& RLO & 19.59 & 0.48 & 0.38 \\
&& DZ & 20.27 & 0.57 & 0.33 \\
&& ZWN & 20.45 & 0.61 & 0.25 \\
&& MG & 20.28 & 0.55 & 0.32 \\
& & DEQ-KL & 20.65 & 0.57 & 0.32 \\
& & DEQ-AA & \textbf{21.45} & \textbf{0.64} & \textbf{0.23} \\
\cmidrule(lr){2-6}

& \multirow{7}{*}{10}
& AA & 20.92 & 0.58 & 0.34 \\
&& RLO & 20.98 & 0.58 & 0.33 \\
&& DZ & 21.18 & 0.60 & 0.30 \\
&& ZWN & 21.79 & 0.66 & 0.23 \\
&& MG & 21.54 & 0.63 & 0.28 \\
& & DEQ-KL & 22.07 & 0.70 & 0.21 \\
& & DEQ-AA & \textbf{22.69} & \textbf{0.72} & \textbf{0.18} \\
\midrule

\multirow{14.5}{*}{Leaves}
& \multirow{7}{*}{4}
& AA & 16.46 & 0.44 & 0.51 \\
&& RLO & 16.47 & 0.43 & 0.57 \\
&& DZ & 16.93 & 0.46 & 0.54 \\
&& ZWN & 17.57 & 0.56 & 0.40 \\
&& MG & 17.62 & 0.57 & \textbf{0.36} \\
& & DEQ-KL & 17.81 & 0.60 & 0.37 \\
& & DEQ-AA & \textbf{18.07} & \textbf{0.61} & \textbf{0.36} \\
\cmidrule(lr){2-6}

& \multirow{7}{*}{10}
& AA & 17.93 & 0.51 & 0.54 \\
&& RLO & 17.77 & 0.53 & 0.46 \\
&& DZ & 18.14 & 0.54 & 0.47 \\
&& ZWN & 17.87 & 0.65 & 0.30 \\
&& MG & 18.47 & 0.62 & 0.31 \\
& & DEQ-KL & 19.91 & 0.71 & 0.28 \\
& & DEQ-AA & \textbf{20.02} & \textbf{0.72} & \textbf{0.24} \\
\bottomrule
\end{tabular}
\end{table}

\begin{table}[htbp]
\centering
\caption{Comparison of PSNR, SSIM, and LPIPS achieved by different restoration methods on grayscale test images degraded by motion blur and noise under various noise levels.}
\label{tab:comparison_gray_motion_metrics}
\begin{tabular}{ccclccc}
\toprule
Image & $L$ & Method & PSNR$\uparrow$ & SSIM$\uparrow$ & LPIPS$\downarrow$ \\
\midrule

\multirow{14.5}{*}{Dart}
& \multirow{7}{*}{4}
& AA & 12.13 & 0.49 & 0.51 \\
&& RLO & 14.09 & 0.51 & 0.38 \\
&& DZ & 15.22 & 0.56 & 0.40 \\
&& ZWN & 17.20 & 0.69 & 0.22 \\
&& MG & 16.35 & 0.68 & 0.22 \\
& & DEQ-KL & \textbf{17.72} & \textbf{0.83} & \textbf{0.17} \\
& & DEQ-AA & 17.64 & \textbf{0.83} & \textbf{0.17} \\
\cmidrule(lr){2-6}

& \multirow{7}{*}{10}
& AA & 15.75 & 0.59 & 0.32 \\
&& RLO & 15.88 & 0.56 & 0.34 \\
&& DZ & 17.37 & 0.79 & 0.17 \\
&& ZWN & 17.97 & \textbf{0.88} & \textbf{0.11} \\
&& MG & 16.87 & 0.78 & 0.17 \\
& & DEQ-KL & 18.17 & 0.83 & 0.14 \\
& & DEQ-AA & \textbf{18.77} & 0.85 & 0.12 \\
\midrule

\multirow{14.5}{*}{SAR}
& \multirow{7}{*}{4}
& AA & 20.03 & 0.53 & 0.37 \\
&& RLO & 19.94 & 0.51 & 0.37 \\
&& DZ & 20.69 & 0.62 & 0.31 \\
&& ZWN & 20.52 & 0.63 & \textbf{0.25} \\
&& MG & 21.33 & 0.63 & 0.26 \\
& & DEQ-KL & 20.77 & 0.62 & 0.39 \\
& & DEQ-AA & \textbf{21.53} & \textbf{0.65} & 0.28 \\
\cmidrule(lr){2-6}

& \multirow{7}{*}{10}
& AA & 21.57 & 0.63 & 0.28 \\
&& RLO & 21.64 & 0.63 & 0.29 \\
&& DZ & 22.02 & 0.66 & 0.28 \\
&& ZWN & 22.63 & 0.72 & 0.21 \\
&& MG & 22.34 & 0.69 & 0.24 \\
& & DEQ-KL & 22.69 & 0.72 & 0.25 \\
& & DEQ-AA & \textbf{23.16} & \textbf{0.74} & \textbf{0.20} \\
\midrule

\multirow{14.5}{*}{Leaves}
& \multirow{7}{*}{4}
& AA & 16.42 & 0.44 & 0.56 \\
&& RLO & 16.49 & 0.45 & 0.52 \\
&& DZ & 16.94 & 0.48 & 0.59 \\
&& ZWN & 17.51 & 0.53 & 0.47 \\
&& MG & 17.84 & 0.59 & \textbf{0.34} \\
& & DEQ-KL & 18.30 & 0.63 & 0.41 \\
& & DEQ-AA & \textbf{18.40} & \textbf{0.64} & 0.38 \\
\cmidrule(lr){2-6}

& \multirow{7}{*}{10}
& AA & 17.32 & 0.55 & 0.47 \\
&& RLO & 18.41 & 0.56 & 0.45 \\
&& DZ & 18.60 & 0.58 & 0.43 \\
&& ZWN & 18.44 & 0.68 & 0.34 \\
&& MG & 19.16 & 0.67 & \textbf{0.28} \\
& & DEQ-KL & 19.77 & 0.67 & 0.34 \\
& & DEQ-AA & \textbf{20.01} & \textbf{0.73} & 0.30 \\
\bottomrule
\end{tabular}
\end{table}

Figures \ref{fig:comparison_gray_Gaussian_meidum_Dart}-\ref{fig:comparison_gray_Motion_low_Dart} present representative restoration results on the grayscale test images. Overall, the proposed method consistently produces superior visual quality compared with the variational-based baselines, particularly for SAR and natural images. For example, in Fig. \ref{fig:comparison_gray_Gaussian_meidum_SAR}, our method recovers sharper road structures while effectively suppressing isolated dark artifacts. Similarly, the fine leaf textures are more faithfully reconstructed in Fig. \ref{fig:comparison_gray_Gaussian_low_Leaves}. For the synthetic Dart image, however, variational methods remain competitive under severe degradations, especially the ZWN model. As shown in Fig. \ref{fig:comparison_gray_Gaussian_meidum_Dart}, ZWN produces higher-contrast dartboard patterns, whereas our method achieves more effective deblurring. Under a lower noise level, the advantage of the proposed method becomes more pronounced. As illustrated in Fig. \ref{fig:comparison_gray_Motion_low_Dart}, our reconstruction preserves sharper cross patterns in the four corners as well as finer structures around the dartboard center.

Quantitative comparisons are reported in Tables \ref{tab:comparison_gray_Gaussian_metrics} and \ref{tab:comparison_gray_motion_metrics} for Gaussian and motion blur, respectively, in terms of PSNR, SSIM, and LPIPS, where the best result in each setting is highlighted in bold. For the synthetic Dart image with Gaussian blur and noise level $L=4$, the proposed method achieves a slightly lower PSNR than the ZWN model, while yielding superior SSIM and LPIPS values, indicating a better balance between structural fidelity and perceptual quality. Apart from this single case, our method consistently outperforms all competing variational models on the remaining grayscale test images. In particular, for the challenging natural image Leaves, the proposed method improves the PSNR by at least 0.46 dB over the best competing baseline under all considered degradation settings.

Finally, DEQ-AA consistently achieves slightly better performance than DEQ-KL across the test set, as reflected, for example, by the quantitative results on Leaves in Tables \ref{tab:comparison_gray_Gaussian_metrics} and \ref{tab:comparison_gray_motion_metrics}. A possible explanation is that the joint deblurring and Gamma denoising task considered in this work exhibits a more complex degradation model, under which the AA-type fidelity provides a more accurate statistical characterization of Gamma noise than the KL-type fidelity. A deeper theoretical and experimental investigation of this observation is left for future work.

\subsection{Color image restoration}
	We evaluate the proposed method on a diverse collection of color images, including 3 images from Set3C \footnote{\url{https://huggingface.co/datasets/deepinv/set3c}}, 24 from Kodak24 \footnote{\url{https://www.kaggle.com/datasets/drxinchengzhu/kodak24}}, 17 from BSD500 \cite{arbelaez2010contour}, and 11 from Set14 \cite{zeyde2010single}. Degraded observations are generated by first applying either an $15\times15$ Gaussian blur kernel with standard deviation $2$ or a realistic motion blur kernel \cite{levin2009understanding}, followed by multiplicative Gamma noise with equivalent numbers of looks $L=4$ and $L=10$.

Similar to the grayscale restoration setting, a method for joint deblurring and multiplicative Gamma denoising of color images has been proposed in \cite{wang2021color}. 
Unfortunately, its source code is not publicly available, preventing us from including it in our experimental comparison.
As a representative learning-based baseline, we therefore adopt the DEQ framework of \cite{daniele2026deep}. Although originally designed for Poisson image restoration, its KL-type data-fidelity model has also been successfully applied to multiplicative Gamma noise removal \cite{steidl2010removing,zhang2022image}, and coincides with the fidelity term employed in our DEQ-KL model. To ensure a fair comparison, we regenerate the training data according to the degradation model adopted in this work and retrain the network of \cite{daniele2026deep} using the corresponding blurred Gamma-noise observations. The retrained network follows the DEQ framework and adopts a Regularization by Denoising (RED) parameterization; accordingly, we denote it by DEQ-RED. The principal difference between DEQ-RED and DEQ-KL lies in the regularization model: DEQ-RED employs an implicit CNN regularizer, whereas DEQ-KL incorporates an explicit geometry-aware regularizer together with the proposed positivity-preserving backtracking strategy.

\begin{figure}[h]
    \centering
 \begin{minipage}{0.3\textwidth}
        \centering
        \subfigure[Clean]{
            \includegraphics[width=\linewidth]{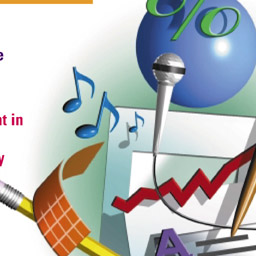}
        }

        \vspace{-1.5mm}
        {\scriptsize (PSNR$\uparrow$, SSIM$\uparrow$, LPIPS$\downarrow$)}
    \end{minipage}
    \hspace{0.02\textwidth}
    \begin{minipage}{0.3\textwidth}
        \centering
        \subfigure[Blurred \& Noisy]{
            \includegraphics[width=\linewidth]{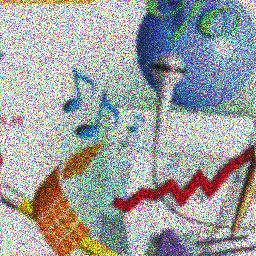}
        }

        \vspace{-1.5mm}
        {\scriptsize (8.25,0.10,1.25)}
    \end{minipage}

    \vspace{0.5em}

 \begin{minipage}{0.3\textwidth}
        \centering
        \subfigure[DEQ-RED \cite{daniele2026deep}]{
            \includegraphics[width=\linewidth]{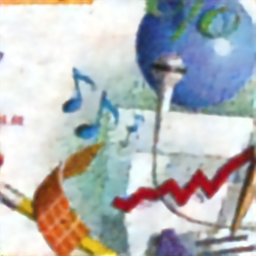}
        }

        \vspace{-2mm}
        {\scriptsize (22.17, 0.75, 0.36)}
    \end{minipage}
    \hspace{0.02\textwidth}
    \begin{minipage}{0.3\textwidth}
        \centering
        \subfigure[DEQ-KL]{
            \includegraphics[width=\linewidth]{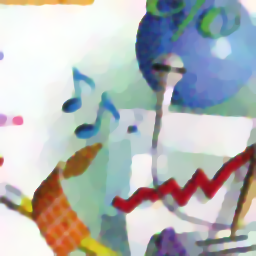}
        }

        \vspace{-2mm}
        {\scriptsize (22.25, 0.78, 0.24)}
    \end{minipage}
    \hspace{0.02\textwidth}
    \begin{minipage}{0.3\textwidth}
        \centering
        \subfigure[DEQ-AA]{
            \includegraphics[width=\linewidth]{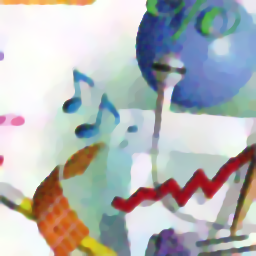}
        }

        \vspace{-2mm}
        {\scriptsize (22.26, 0.79, 0.24)}
    \end{minipage}
    \caption{Comparison of restoration results obtained by different methods on the Magazine image degraded by Gaussian blur and noise with noise level $L=4$.}
    \label{fig:comparison_color_Gaussian_meidum_Magazine}
\end{figure}

\begin{figure}[h]
    \centering
 \begin{minipage}{0.3\textwidth}
        \centering
        \subfigure[Clean]{
            \includegraphics[width=\linewidth]{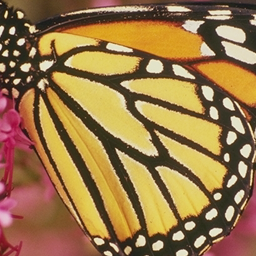}
        }

        \vspace{-1.5mm}
        {\scriptsize (PSNR$\uparrow$, SSIM$\uparrow$, LPIPS$\downarrow$)}
    \end{minipage}
    \hspace{0.02\textwidth}
    \begin{minipage}{0.3\textwidth}
        \centering
        \subfigure[Blurred \& Noisy]{
            \includegraphics[width=\linewidth]{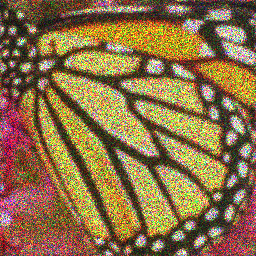}
        }

        \vspace{-1.5mm}
        {\scriptsize (11.47,0.23,0.79)}
    \end{minipage}

    \vspace{0.5em}

 \begin{minipage}{0.3\textwidth}
        \centering
        \subfigure[DEQ-RED \cite{daniele2026deep}]{
            \includegraphics[width=\linewidth]{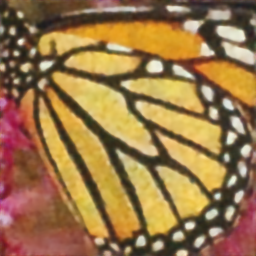}
        }

        \vspace{-2mm}
        {\scriptsize (21.53, 0.72, 0.27)}
    \end{minipage}
    \hspace{0.02\textwidth}
    \begin{minipage}{0.3\textwidth}
        \centering
        \subfigure[DEQ-KL]{
            \includegraphics[width=\linewidth]{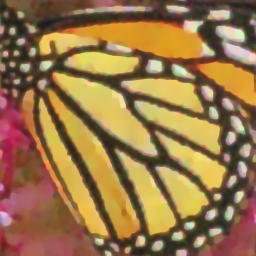}
        }

        \vspace{-2mm}
        {\scriptsize (21.62, 0.74, 0.23)}
    \end{minipage}
    \hspace{0.02\textwidth}
    \begin{minipage}{0.3\textwidth}
        \centering
        \subfigure[DEQ-AA]{
            \includegraphics[width=\linewidth]{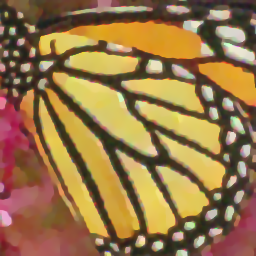}
        }

        \vspace{-2mm}
        {\scriptsize (21.74, 0.75, 0.18)}
    \end{minipage}
    \caption{Comparison of restoration results obtained by different methods on the Butterfly image degraded by Gaussian blur and noise with noise level $L=4$.}
    \label{fig:comparison_color_Gaussian_meidum_Butterfly}
\end{figure}

\begin{figure}[h]
    \centering
 \begin{minipage}{0.3\textwidth}
        \centering
        \subfigure[Clean]{
            \includegraphics[width=\linewidth]{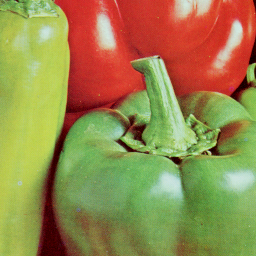}
        }

        \vspace{-2mm}
        {\scriptsize (PSNR$\uparrow$, SSIM$\uparrow$, LPIPS$\downarrow$)}
    \end{minipage}
    \hspace{0.02\textwidth}
    \begin{minipage}{0.3\textwidth}
        \centering
        \subfigure[Blurred \& Noisy]{
            \includegraphics[width=\linewidth]{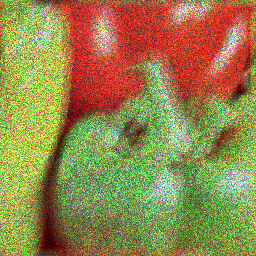}
        }

        \vspace{-2mm}
        {\scriptsize (14.62,0.09,0.94)}
    \end{minipage}

    \vspace{0.5em}

 \begin{minipage}{0.3\textwidth}
        \centering
        \subfigure[DEQ-RED \cite{daniele2026deep}]{
            \includegraphics[width=\linewidth]{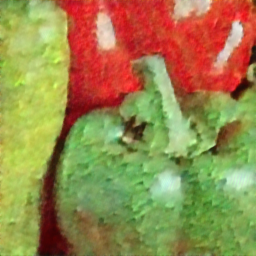}
        }

	\vspace{-2mm}
        {\scriptsize (24.11, 0.53, 0.43)}
    \end{minipage}
    \hspace{0.02\textwidth}
    \begin{minipage}{0.3\textwidth}
        \centering
        \subfigure[DEQ-KL]{
            \includegraphics[width=\linewidth]{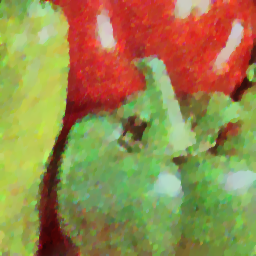}
        }

        \vspace{-2mm}
        {\scriptsize (24.12, 0.55, 0.36)}
    \end{minipage}
    \hspace{0.02\textwidth}
    \begin{minipage}{0.3\textwidth}
        \centering
        \subfigure[DEQ-AA]{
            \includegraphics[width=\linewidth]{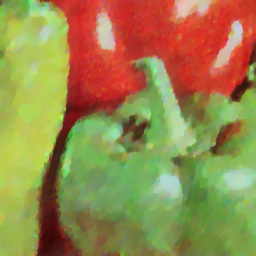}
        }

        \vspace{-2mm}
        {\scriptsize (24.52, 0.60, 0.30)}
    \end{minipage}
    \caption{Comparison of restoration results obtained by different methods on the Peppers image degraded by motion blur and noise with noise level $L=10$.}
    \label{fig:comparison_color_motion_low_Peppers}
\end{figure}

\begin{table}[htbp]
\centering
\caption{Average quantitative results (PSNR, SSIM, and LPIPS) on the color test set under different blur types and noise levels.}
\label{tab:comparison_metrics_color}

\renewcommand{\arraystretch}{1.25}

\begin{tabular}{c l c c c}
\toprule
$L$ & Method & PSNR$\uparrow$ & SSIM$\uparrow$ & LPIPS$\downarrow$ \\
\midrule

\multicolumn{5}{c}{\textbf{Gaussian Blur}}\\
\midrule

\multirow{3}{*}{4}
& DEQ-RED \cite{daniele2026deep}& 22.87 & 0.58 & 0.55 \\
& DEQ-KL  & 23.09 & 0.59 & 0.48 \\
& DEQ-AA  & \textbf{23.17} & \textbf{0.60} & \textbf{0.47} \\
\cmidrule(lr){1-5}

\multirow{3}{*}{10}
& DEQ-RED \cite{daniele2026deep}& 23.67 & 0.60 & 0.51 \\
& DEQ-KL  & 23.79 & 0.61 & 0.44 \\
& DEQ-AA  & \textbf{24.03} & \textbf{0.63} & \textbf{0.41} \\

\midrule

\multicolumn{5}{c}{\textbf{Motion Blur}}\\
\midrule

\multirow{3}{*}{4}
& DEQ-RED \cite{daniele2026deep}& 20.55 & 0.46 & 0.63 \\
& DEQ-KL  & 20.62 & 0.48 & 0.60 \\
& DEQ-AA  & \textbf{20.99} & \textbf{0.49} & \textbf{0.55} \\
\cmidrule(lr){1-5}

\multirow{3}{*}{10}
& DEQ-RED \cite{daniele2026deep}& 21.08 & 0.49 & 0.58 \\
& DEQ-KL  & 21.92 & \textbf{0.54} & 0.49 \\
& DEQ-AA  & \textbf{21.97} & \textbf{0.54} & \textbf{0.48} \\
\bottomrule
\end{tabular}
\end{table}

Table \ref{tab:comparison_metrics_color} reports the quantitative results of the proposed methods and the competing approaches on the color image test set, where the best performance in each setting is highlighted in bold. Under both Gaussian and motion blur degradations, DEQ-KL achieves performance comparable to that of DEQ-RED, indicating that incorporating an explicit variational regularizer does not compromise restoration quality relative to implicit regularization. A further advantage of the proposed framework is its parameter efficiency. Owing to the incorporation of explicit geometric priors, our model contains only $8\times 24+8\times 33+3\times 33 +2\approx 600$ trainable parameters, whereas DEQ-RED requires approximately $10^5$ parameters. This substantial reduction in model complexity is expected to improve training efficiency while maintaining competitive performance. Moreover, DEQ-AA consistently outperforms DEQ-KL across most test cases, which is consistent with the observations from the grayscale image experiments.

Figures \ref{fig:comparison_color_Gaussian_meidum_Magazine}-\ref{fig:comparison_color_motion_low_Peppers} present representative visual comparisons. Under Gaussian blur, the proposed DEQ-AA method reconstructs sharper geometric structures and finer textures, as illustrated by the butterfly patterns in Figure \ref{fig:comparison_color_Gaussian_meidum_Butterfly}.

Another notable observation is that DEQ-RED, which relies on an implicit neural prior for image reconstruction, occasionally introduces visually noticeable artifacts. This phenomenon is also reflected by its relatively inferior LPIPS scores reported in Table \ref{tab:comparison_metrics_color}. For example, although DEQ-RED and the proposed method achieve comparable PSNR values in the Gaussian blur experiment shown in Figure \ref{fig:comparison_color_Gaussian_meidum_Magazine}, DEQ-RED produces visible shadow-like artifacts that degrade perceptual quality. The difference becomes more pronounced under the more challenging real motion blur degradation. As illustrated in Figure \ref{fig:comparison_color_motion_low_Peppers}, DEQ-RED generates noticeable blotchy artifacts in the Peppers image, whereas the proposed method preserves cleaner structures and more natural visual appearance. We attribute this behavior to the explicit geometry-aware variational regularization, which provides stronger structural priors and effectively suppresses the artifacts that may arise in purely data-driven restoration models.

\subsection{On the convergence of DEQ models at inference}
In all numerical experiments, the iterative solver is terminated once the relative error satisfies the stopping criterion
\begin{equation*}
\frac{\|x^{k+1}-x^k\|}{\|x^{k+1}\|}\le \epsilon_S.
\end{equation*}
To better illustrate the convergence behavior of the proposed method, we adopt a stringent stopping tolerance of $\epsilon_S=10^{-6}$ in this subsection. In all other experiments, we use the empirical choice $\epsilon_S=10^{-5}$, which provides a good compromise between reconstruction accuracy and computational efficiency.

To illustrate the convergence behavior of the proposed DEQ models, we report the evolution of the objective value $\Psi$ and the PSNR for image Magazine under Gaussian blur with Gamma noise of level $L=4$; the corresponding results for DEQ-KL and DEQ-AA are shown in Fig. \ref{fig:Convergence}. It can be observed that the objective values of both models decrease monotonically and eventually stabilize. In terms of reconstruction quality, the PSNR of DEQ-AA reaches a stable value after approximately 700 iterations, whereas DEQ-KL requires about 1000 iterations to converge. This observation indicates that the minimization of the proposed objective functional is well aligned with improvements in reconstruction quality, suggesting that convergence of the optimization process is accompanied by convergence to a high-quality reconstruction.

Furthermore, since a backtracking strategy is employed, the step size is adaptively adjusted during the iterations, which leads to discontinuities in the convergence curves of both DEQ-KL and DEQ-AA. The iterations at which the step size is updated are indicated by red vertical lines in Fig. \ref{fig:Convergence}. It can be observed that the objective energy continues to decrease and the PSNR keeps increasing after these adjustments, demonstrating the effectiveness of the adopted backtracking strategy.
\begin{figure}[htbp]
    \centering
    \begin{minipage}{\textwidth}
        \centering
        \subfigcapskip=1pt
        \subfigure[DEQ-KL]{
            \includegraphics[width=0.48\textwidth]{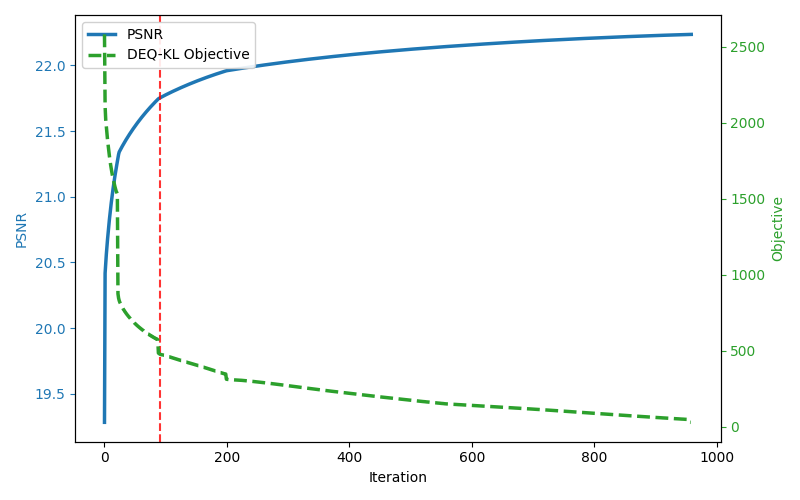}
        }
        \subfigure[DEQ-AA]{
            \includegraphics[width=0.48\textwidth]{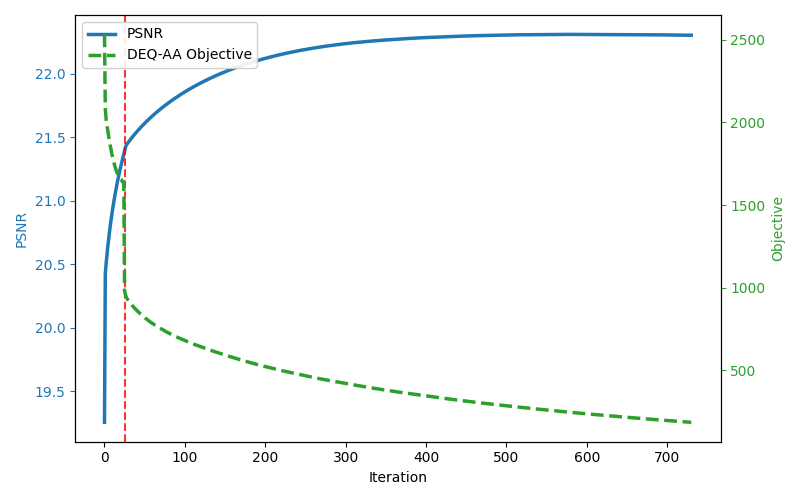}
        }
    \end{minipage}
  \caption{Convergence verification of (a) DEQ-KL and (b) DEQ-AA methods, showing PSNR and energy variations with iterations.}
    \label{fig:Convergence}
\end{figure}
\subsection{Learned filters and characteristic influence functions in neural networks}

One of the main contributions of the developed DEQ model is the possibility to explain the reconstructions as critical points of an energy functional with interpretable learned regularization term \eqref{eq:regular1}.
To gain further insight into the learned image priors, we visualize representative convolution filters and influence functions in this subsection.

Fig. \ref{fig:learned_filters} compares the convolution filters before and after training, where the color intensity indicates the magnitude and sign of the filter coefficients. Although the randomly initialized filters exhibit no discernible structure, the learned filters evolve into organized patterns that capture meaningful local image features. To further reveal the learned representation, Fig. \ref{fig:psi_evolution} visualizes the responses
$\sum_{l=1}^{N_a}\psi_l^a(k_l*y)$, obtained by applying the learned filters and influence functions to the degraded observation $y$ of the grayscale \emph{Parrot} image corrupted by Gaussian blur and multiplicative Gamma noise with $L=4$. As training progresses, these responses become progressively closer to the corresponding ground-truth image, indicating that the learned representation effectively captures the underlying image structure.
\begin{figure}[htbp]
\centering
\setlength{\tabcolsep}{2pt}

\begin{tabular}{cccc@{\hspace{1cm}}cccc}

\multicolumn{4}{c}{\textbf{Initialization}} &
\multicolumn{4}{c}{\textbf{After training}} \\[2mm]

\includegraphics[width=0.10\textwidth]{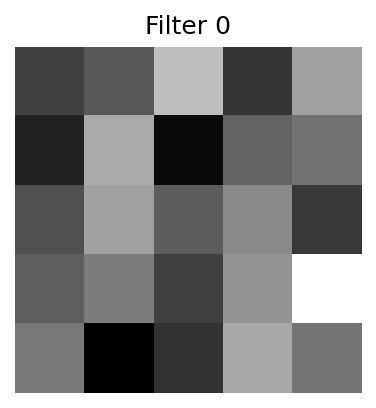} &
\includegraphics[width=0.10\textwidth]{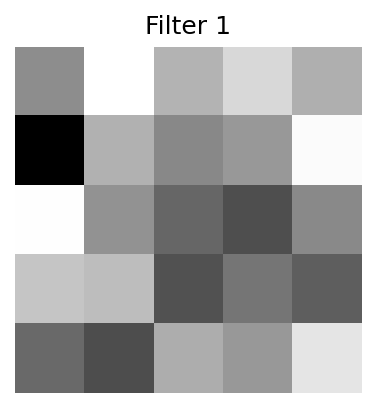} &
\includegraphics[width=0.10\textwidth]{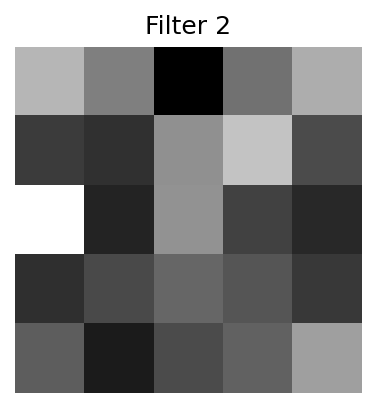} &
\includegraphics[width=0.10\textwidth]{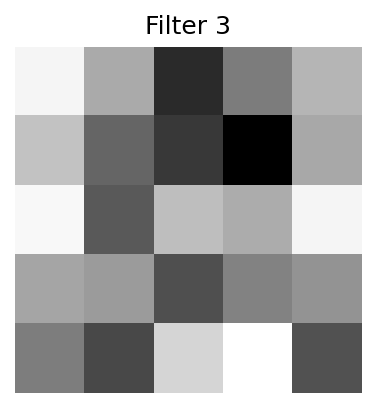} &
\includegraphics[width=0.10\textwidth]{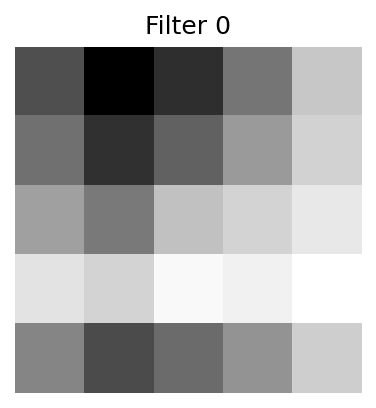} &
\includegraphics[width=0.10\textwidth]{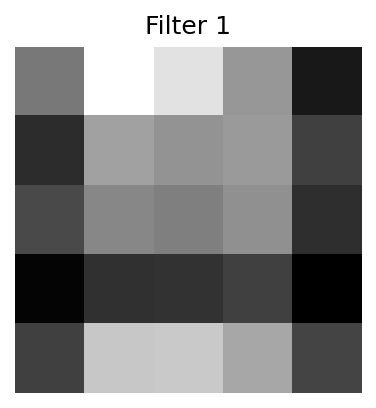} &
\includegraphics[width=0.10\textwidth]{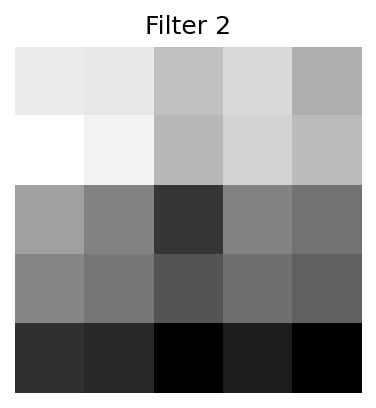} &
\includegraphics[width=0.10\textwidth]{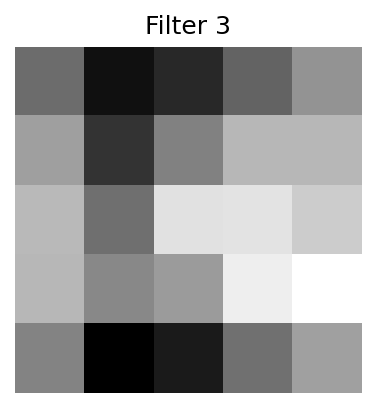} \\

\includegraphics[width=0.10\textwidth]{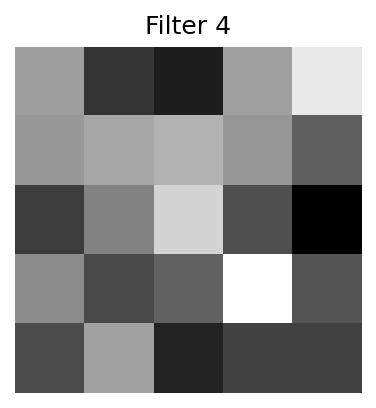} &
\includegraphics[width=0.10\textwidth]{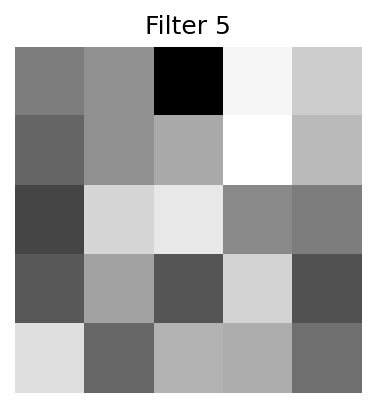} &
\includegraphics[width=0.10\textwidth]{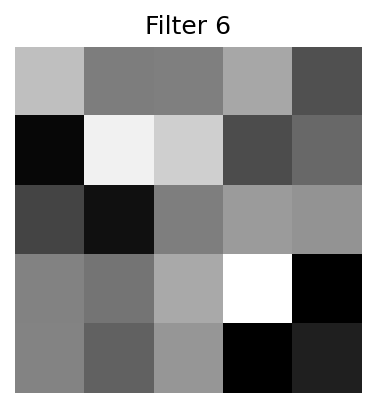} &
\includegraphics[width=0.10\textwidth]{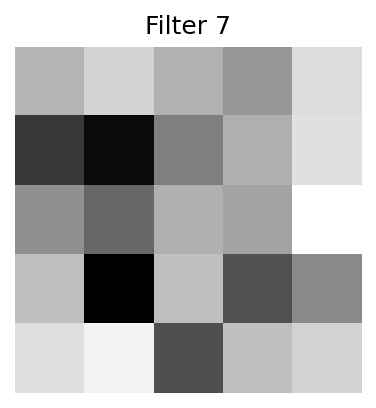} &
\includegraphics[width=0.10\textwidth]{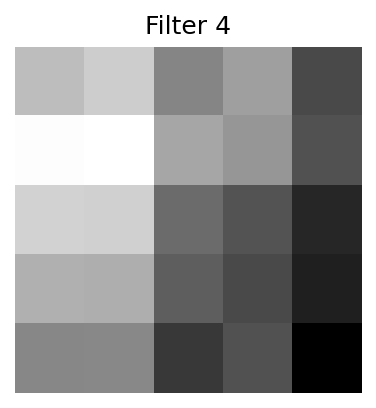} &
\includegraphics[width=0.10\textwidth]{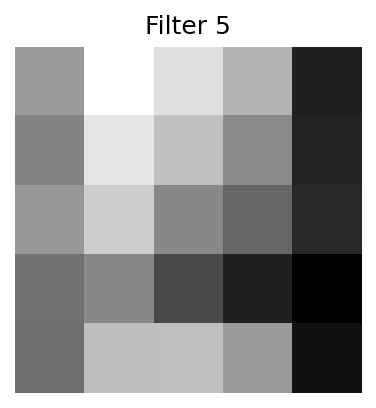} &
\includegraphics[width=0.10\textwidth]{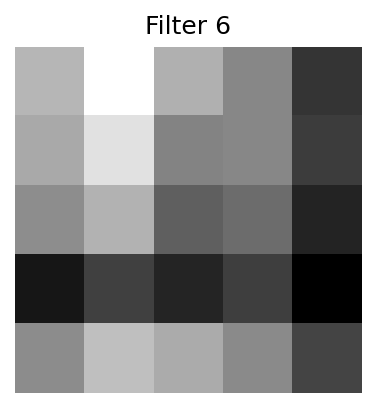} &
\includegraphics[width=0.10\textwidth]{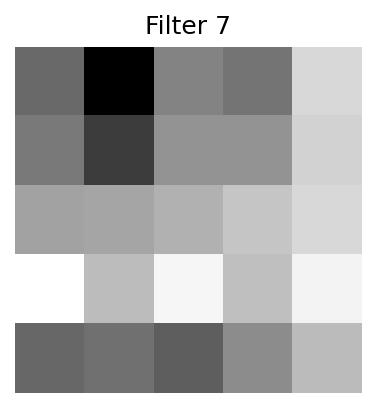} \\

\includegraphics[width=0.10\textwidth]{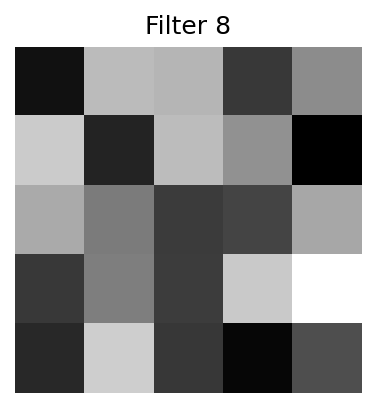} &
\includegraphics[width=0.10\textwidth]{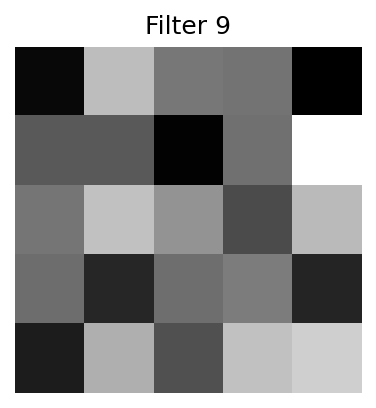} &
\includegraphics[width=0.10\textwidth]{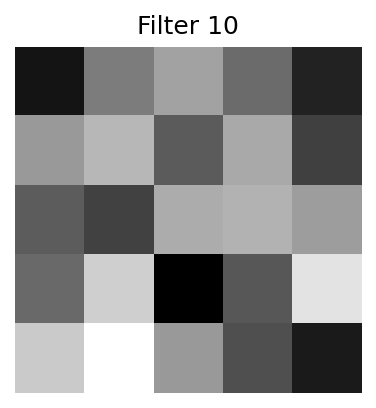} &
\includegraphics[width=0.10\textwidth]{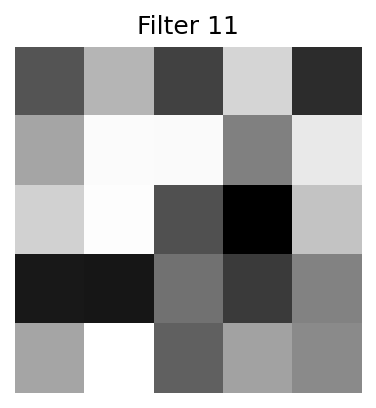} &
\includegraphics[width=0.10\textwidth]{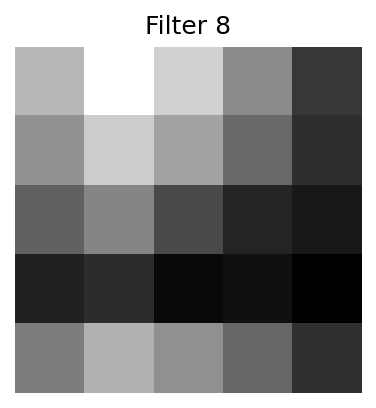} &
\includegraphics[width=0.10\textwidth]{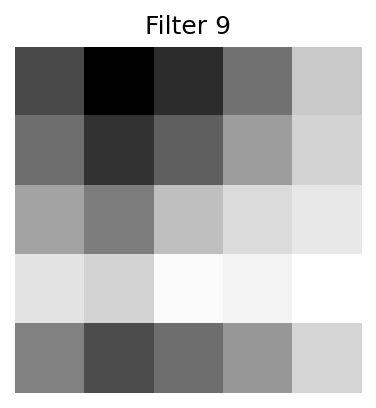} &
\includegraphics[width=0.10\textwidth]{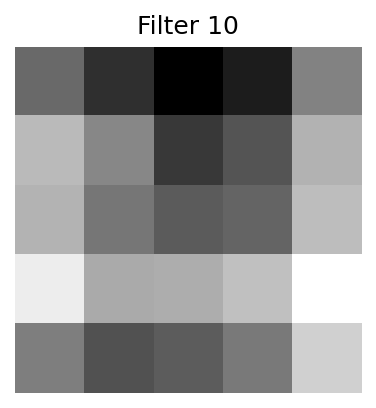} &
\includegraphics[width=0.10\textwidth]{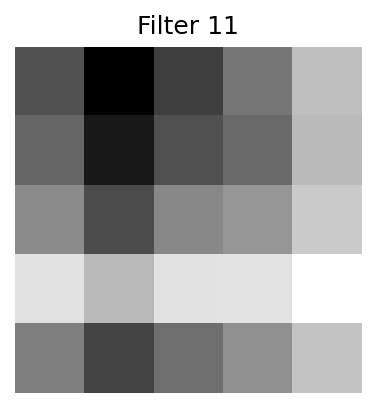}

\end{tabular}

\caption{Representative convolution filters before and after training.}
\label{fig:learned_filters}
\end{figure}

\begin{figure*}[t]
\centering

\hfill
\begin{minipage}{0.16\textwidth}
\centering
\vspace{-0.6mm}
\includegraphics[width=\linewidth]{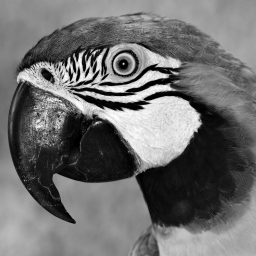}
{\footnotesize Clean}
\end{minipage}
\begin{minipage}{0.16\textwidth}
\centering
\includegraphics[width=\linewidth]{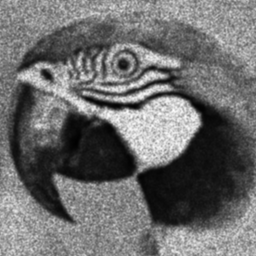}

{\footnotesize Indicator \cite{zhang2022image}}
\end{minipage}
\hfill
\begin{minipage}{0.16\textwidth}
\centering
\includegraphics[width=\linewidth]{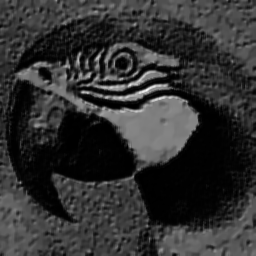}

{\footnotesize Epoch 0}
\end{minipage}
\hfill
\begin{minipage}{0.16\textwidth}
\centering
\includegraphics[width=\linewidth]{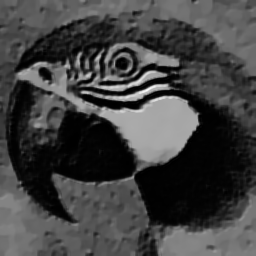}

{\footnotesize Epoch 5}
\end{minipage}
\hfill
\begin{minipage}{0.16\textwidth}
\centering
\includegraphics[width=\linewidth]{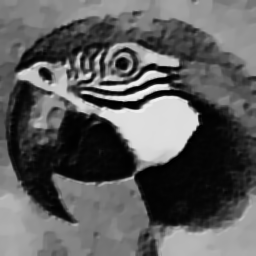}
{\footnotesize Final epoch}
\end{minipage}
\hfill
\begin{minipage}{0.16\textwidth}
\centering
\includegraphics[width=\linewidth]{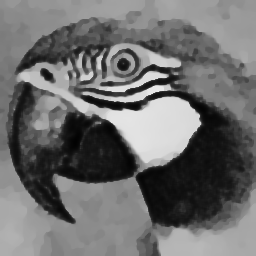}

{\footnotesize Recovery}
\end{minipage}
\caption{
Evolution of the learned area-related influence function throughout training and the corresponding reconstruction performance. From left to right: the ground-truth \emph{Parrot} image; the gray-level indicator $\left(\frac{G_\sigma *y}{\max G_\sigma *y }\right)^p$ from \cite{zhang2022image}, where the recommended parameters $\sigma=0.8$, $p=1.2$ are adopted; the learned area-related indicator $\sum_{l=1}^{N_a}\psi_l^a(k_l*y)$ at epochs 0, 5, and the final epoch; and the reconstruction produced by the final learned regularizer.}
\label{fig:psi_evolution}
\end{figure*}

This observation is consistent with the purpose of the area-related component in the proposed regularizer. Since the variance of multiplicative Gamma noise is proportional to the underlying signal intensity, brighter regions are generally subject to stronger noise corruption. Therefore, the quantity
$
\sum_{l=1}^{N_a}\psi_l^a(k_l*y)
$
is expected to provide a rough estimate of the latent clean image, 
enabling the strength of the area-related regularization to be adapted according to the local image content. A related strategy was adopted in \cite{zhang2022image,yang2025mixed}, where the normalized Gaussian-smoothed image, $\left(\frac{G_\sigma *y}{\max G_\sigma *y }\right)^p$,
was employed as a gray-level indicator. As illustrated in Fig. \ref{fig:psi_evolution}, the learned gray-level indicator provides a substantially more accurate approximation of the clean image than the fixed Gaussian indicator. These results demonstrate the benefit of learning the convolution filters and influence functions directly from data, thereby providing further support for the proposed extension of the model in \cite{yang2025mixed}.

The learned influence functions subsequently perform nonlinear modulation of these extracted local features. As shown in Fig. \ref{fig:influence_filt}(a), the curvature-related penalty function $\psi^c$ consistently attains its minimum around zero curvature, indicating that the proposed model favors locally flat or straight image structures while suppressing unnecessary geometric oscillations. This observation is consistent with the geometric interpretation of curvature regularization reported in \cite{zhu2012image}, where zero curvature corresponds to locally smooth structures.

In contrast, the influence functions $\psi^a$ associated with the learned convolutional filters exhibit substantially richer behaviors, as illustrated in Figs. \ref{fig:influence_filt}(b)-\ref{fig:influence_filt}(d). Besides the conventional unimodal penalties centered at zero, double-well and multi-well potentials are frequently observed, which have been widely adopted in image restoration models \cite{bertozzi2016diffuse,li2011multiphase}. This suggests that different filters favor different response magnitudes rather than uniformly penalizing all nonzero responses. Since different filters capture distinct local image characteristics, the corresponding learned penalties naturally adapt to these heterogeneous statistics. Moreover, several learned influence functions are asymmetric, implying that positive and negative filter responses are treated differently and thereby providing greater flexibility for modeling directional image structures.

We conduct an additional experiment to investigate the dependence of the learned regularizer $R_\theta$ on the training data. Specifically, we train the network on a dataset of piecewise-constant images with smooth boundaries, for which we expect the curvature component of $R_\theta$ to play a less significant role. The dataset consists of 300 images containing ellipses with random positions, aspect ratios, orientations, and intensities, generated using the Operator Discretization Library \cite{adler2017operator}.

Taking the AA-type fidelity as an example, we found that the learned area-related influence functions consistently exhibit the pattern shown as Activation II in Fig.~\ref{fig:influence_filt}, whereas the curvature-related influence functions remain nearly identical to their quadratic initialization throughout training (Fig.~\ref{fig:actandb_ODL}(a)). Meanwhile, the learned parameter $b$ gradually decreases during training (Fig.~\ref{fig:actandb_ODL}(b)). These results indicate that, for piecewise-constant images, the proposed framework naturally emphasizes the area-related regularization and learns a regularizer resembling the classical TV model.

\begin{figure}[h]
    \centering
    \begin{minipage}{\textwidth}
        \subfigcapskip=1pt
		\centering
        \subfigure[Activation I]{
            \includegraphics[width=0.23\textwidth,trim=0 0 0 17bp, clip]{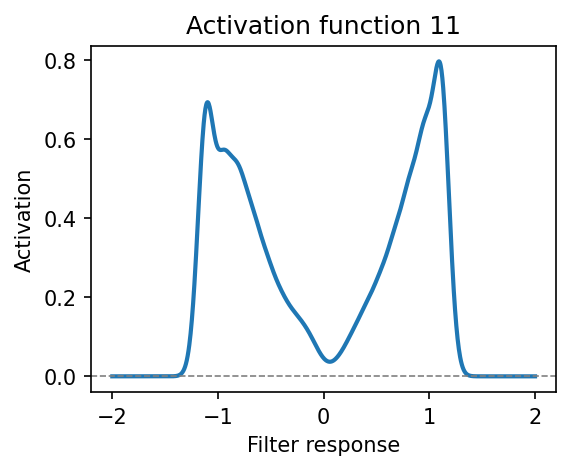}
        }\hspace{-0.2cm}
        \subfigure[Activation II]{
            \includegraphics[width=0.23\textwidth,trim=0 0 0 17bp, clip]{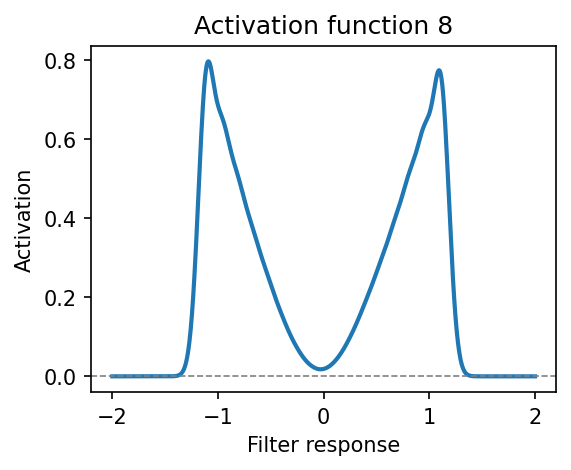}
        }\hspace{-0.2cm}
        \subfigure[Activation III]{
            \includegraphics[width=0.23\textwidth,trim=0 0 0 17bp, clip]{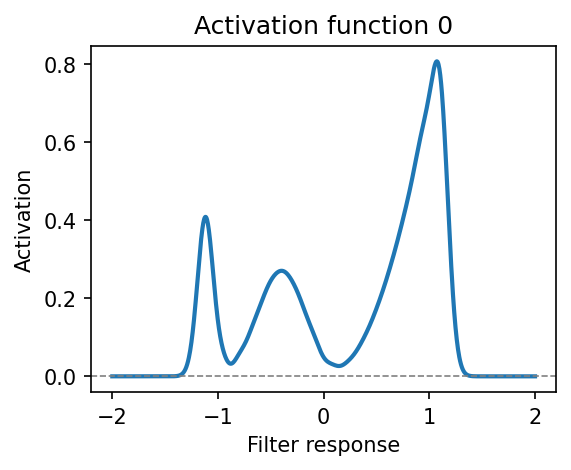}
        }\hspace{-0.2cm}
	     \subfigure[Activation IV]{
            \includegraphics[width=0.23\textwidth,trim=0 0 0 17bp, clip]{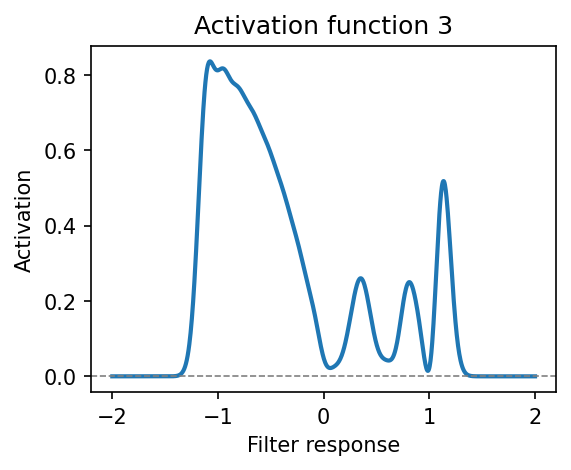}
        }
  \end{minipage}
\caption{Representative learned influence functions.
(a) The learned curvature influence function $\psi^{c}$.
(b)--(d) Representative learned activation influence functions $\psi^{a}$ corresponding to different convolutional filters, illustrating truncated convex, double-well, and multiple-well structures.}
    \label{fig:influence_filt}
\end{figure}

\begin{figure}[h]
    \centering
\begin{minipage}{\textwidth}
    \centering
    \subfigcapskip=1pt
    \subfigure[Curvature-aware influence function]{
        \includegraphics[width=0.43\textwidth,height=5.2cm,trim=0 0 0 17bp, clip]{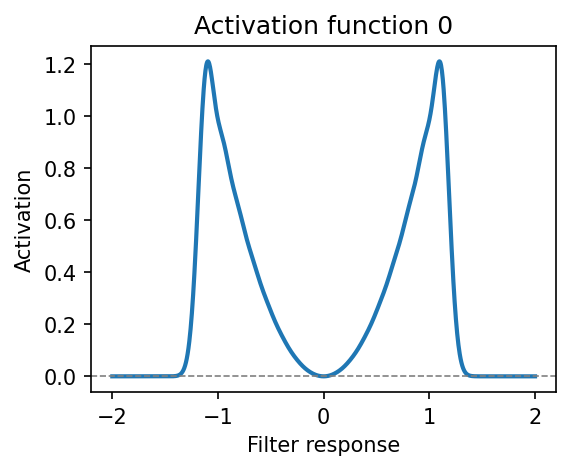}
    }
    \subfigure[Parameter $b$ evolution]{
        \raisebox{2mm}{
            \hspace{0mm}
            \includegraphics[width=0.46\textwidth]{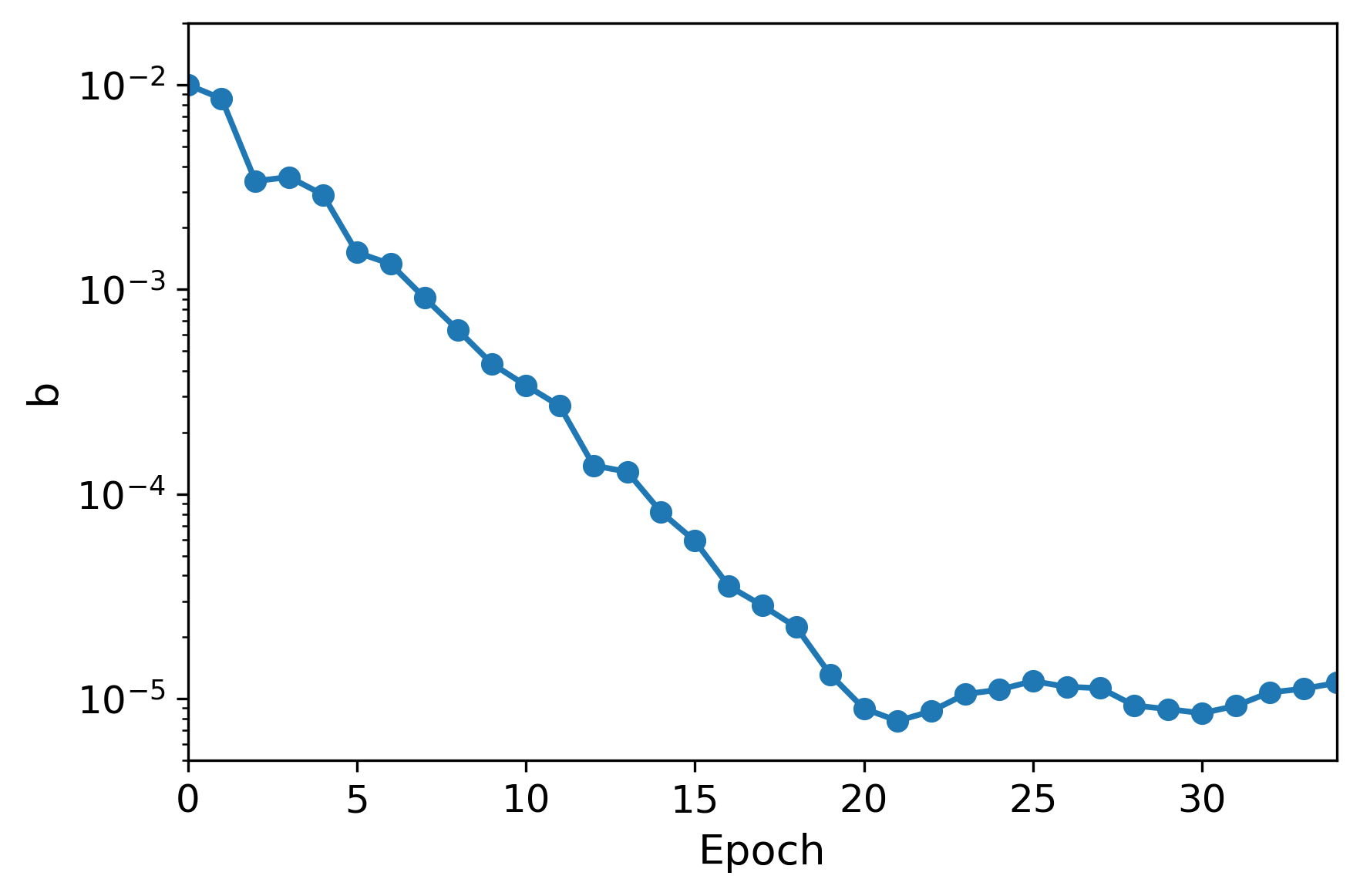}
        }
    }
\end{minipage}
  \caption{Ellipses dataset results. (a) Learned curvature-aware influence function. (b) Evolution of parameter $b$ across training epochs.}
    \label{fig:actandb_ODL}
\end{figure}

\section{Conclusions}\label{sec:conclude}

In this work, we proposed a deep equilibrium framework for joint image deblurring and multiplicative Gamma noise removal. By explicitly incorporating geometry-aware regularization with learnable filters and influence functions, the proposed method provides an interpretable balance between model-based and data-driven approaches. Furthermore, based on a suitable Bregman mirror descent scheme and the $o$-minimal framework, the convergence of the proposed algorithm to a critical point is established. Extensive experiments on grayscale and color images demonstrate the effectiveness of the proposed method compared with existing variational-based, and deep learning approaches.

\appendix
\section{Bregman geometry and proximal operators}\label{Appendix:Bregman}

In this section, we recall the definitions of the fundamental concepts from nonconvex optimization used throughout Sections \ref{ssec:MD_DEQ} and \ref{sec:Cov_DEQ}.

\begin{definition}[Kernel generating distance \cite{auslender2006interior}] Let $\Omega \subseteq \mathbb{R}^{n}$ be a nonempty, open, and convex set. A function $\phi:\mathbb{R}^{n}\rightarrow(-\infty,+\infty]$ is called a \emph{kernel generating distance} associated with $\Omega$ if it satisfies the following conditions:
\begin{enumerate}
    \item[(i)]  $\phi$ is proper, lower semicontinuous, and convex. Moreover, $\operatorname{dom}\phi \subseteq \overline{\Omega}$, $\operatorname{int}(\operatorname{dom}\phi)=\Omega$.
    \item[(ii)]  $\phi$ is continuously differentiable on $\Omega$, i.e., $\phi\in \mathcal{C}^{1}(\Omega)$.
\end{enumerate}
\label{def:kernelgere}
\end{definition}
We denote the class of kernel generating distances by $\mathcal{G}(\Omega)$.

\begin{definition}\cite{rockafellar1997convex} Let $h:X\subset\mathbb{R}^n\rightarrow(-\infty,+\infty]$ be a proper lower semicontinuous function.
\begin{enumerate}
\item[(i)] (essential smoothness). A differentiable function $h$ on $\operatorname{int}(\operatorname{dom}(h))$ is said to be essentially smooth if, for every sequence $\{x^k\}\subset\operatorname{int}(\operatorname{dom}(h))$ converging to a boundary point of $\operatorname{dom}(h)$, we have $\|\nabla h(x^k)\|\to+\infty$.
\item[(ii)](Legendre function). The function $h$ is called a Legendre function if it is essentially smooth and strictly convex on $\operatorname{int}(\operatorname{dom}(h))$.
\end{enumerate}
\label{def:legendre}
\end{definition}
\begin{definition}\cite{bregman1967relaxation}. Given a Legendre function $h\in\mathcal{G}(\Omega):X\subset\mathbb{R}^n\rightarrow(-\infty,+\infty]$, the associated Bregman distance between $x\in\operatorname{dom}h$ and $y\in\operatorname{int}(\operatorname{dom}(h))$ is defined by
\begin{equation*}
D_h(x,y)=h(x)-h(y)-\langle\nabla h(y),x-y\rangle.
\end{equation*}
\label{def:bregmandistance}
\end{definition}
\begin{remark}Owing to the convexity of $h$, the Bregman distance is strictly convex with respect to its first argument. Moreover, $D_h(x,y)\ge 0$ for all $(x,y)\in \operatorname{dom}(h)\times\operatorname{int}(\operatorname{dom}(h))$, with equality if and only if $x=y$. Consequently, $D_h$ serves as a generalized proximity measure. Unlike the Euclidean distance, however, it is generally non-symmetric. The special choice $h=\|\cdot\|^2$ yields the classical squared Euclidean distance.
\end{remark}
\begin{definition} Let $R:X\subset\mathbb{R}^n\rightarrow(-\infty,+\infty]$ be a proper lower semicontinuous function and $h:X \subset\mathbb{R}^n\rightarrow(-\infty,+\infty]$ be a Legendre function. The Bregman proximal operator associated with $R$ (w.r.t. $h$) is defined by
\begin{equation*}
\operatorname{Prox}_{\tau R}^h(x)=\arg\min_{u\in X}\lambda R(u)+D_h(u,x),\;x\in\operatorname{int}(\operatorname{dom}(h)),
\end{equation*}
where $\lambda>0$ is the step-size.
\end{definition}
\begin{remark}
Since $R$ may be nonconvex, the corresponding proximal mapping $\operatorname{Prox}_{\tau R}^h$ is generally set-valued. If, in addition, $R$ is convex and coercive, then the Bregman proximal operator is single-valued \cite{daniele2026deep}.
\end{remark}

\section{Kurdyka-\L{}ojasiewicz functions}\label{Appendix:KLfunctions}

This section specifically focuses on the Kurdyka-\L{}ojasiewicz (K\L{}), which plays a central role in the convergence analysis of nonconvex optimization algorithms. 
We first recall its definition and then discuss its verification in more general settings, such as the one considered in this project, where simpler tools based on global subanalyticity cannot be directly applied.

\begin{definition}Let $\Phi:\mathbb{R}^n\rightarrow (-\infty,+\infty]$ be a proper lower semicontinuous function.
\begin{enumerate}
\item[(a)]The function $\Phi$ is said to satisfy the K\L{} property at a point $\bar{x}\in\operatorname{dom}\Phi$ if there exist $\eta\in(0,+\infty]$, a neighborhood $U$ of $\bar{x}$, and a continuous concave function $\psi:[0,\eta)\rightarrow \mathbb{R}_+$ such that:
\begin{enumerate}
\item[(i)]$\psi(0)=0$;
\item[(ii)]$\psi$ is continuously differentiable on $(0,\eta)$ and continuous at 0;
\item[(ii)]$\psi^\prime(t)>0$ for all $t\in(0,\eta)$;
\item[(iv)]for every $x\in U\cap\{x\in\mathbb{R}^n:\Phi(\bar{x})<\Phi(x)<\Phi(\bar{x})+\eta\}$, the K\L{} inequality
\begin{equation*}
\psi^\prime(\Phi(x)-\Phi(\bar{x}))\operatorname{dist}(0,\partial \Phi(x))\geq 1
\end{equation*}
holds, where $\operatorname{dist}(0,\partial \Phi(x))$ denotes the distance from 0 to the subgradient set $\partial \Phi$.
\end{enumerate}
\item[(b)] The function $\Phi$ is called a K\L{} function if it satisfies the K\L{} property at every point in $\operatorname{dom}\Phi$.
\end{enumerate}
\end{definition}
Although the K\L{} property is generally difficult to verify directly from its definition, it is known to hold for several broad classes of functions frequently encountered in imaging and machine learning. Two particularly important examples are semialgebraic functions and functions definable in an $o$-minimal structure, which we briefly review below.

\subsection{Semialgebraic functions}\label{subsection:semialgic}
We next recall the class of semialgebraic sets and functions.
\begin{definition}(Semialgebraic sets and semialgebraic functions).
\begin{itemize}
\item[$\bullet$] A set $S\subset\mathbb{R}^n$ is called semialgebraic if there exist finitely many polynomial functions $f_{ij},g_{ij}:\mathbb{R}^n\rightarrow \mathbb{R}$ such that
\begin{equation*}
S=\bigcup_{i=1}^p\bigcap_{j=1}^q\{x\in\mathbb{R}^n:f_{ij}(x)=0, g_{ij}(x)<0\}.
\end{equation*}
\item[$\bullet$]A function $f:\mathbb{R}^n\rightarrow\mathbb{R}\cup\{+\infty\}$ is called semialgebraic if its graph
\begin{equation*}
\{(x,t)\in\mathbb{R}^{n+1}:f(x)=t\}
\end{equation*}
is a semialgebraic subset of $\mathbb{R}^{n+1}$.
\end{itemize}
\end{definition}

The class of semialgebraic sets is closed under finite unions, finite intersections, complements, and Cartesian products. Correspondingly, semialgebraic functions form a rich and flexible class, which is closed under finite sums, products, and compositions. Important examples include real polynomial functions and indicator functions of semialgebraic sets.


\subsection{Functions definable in an $o$-minimal structure}
We then recall the notion of an $o$-minimal structure, introduced by Lou van den Dries \cite{van1996geometric}, which provides a powerful framework for analyzing functions arising in optimization; see also \cite{bolte2007clarke}.
\begin{definition}{($o$-minimal structure)}
Let $\mathcal{M}=\{\mathcal{M}_n\}_{n\geq1}$ be a sequence such that each $\mathcal{M}_n$ is a family of subsets of $\mathbb{R}^n$. We say that $\mathcal{M}$ is an $o$-minimal structure on the field $(\mathbb{R},+,\cdot)$ if it satisfies the following conditions:

\begin{enumerate}
    \item[(i)] For every $n\geq1$, the collection $\mathcal{M}_n$ is a Boolean algebra of subsets of $\mathbb{R}^n$. Equivalently, it contains $\emptyset$ and $\mathbb{R}^n$, and is closed under finite unions, finite intersections, and complements.

    \item[(ii)] If $S\in\mathcal{M}_n$, then both cylinder sets
    \[
    S\times\mathbb{R}
    \quad\text{and}\quad
    \mathbb{R}\times S
    \]
    belong to $\mathcal{M}_{n+1}$.

    \item[(iii)] Whenever $T\in\mathcal{M}_{n+1}$, its projection onto the first $n$ coordinates,
    \[
    \pi(T)
    :=
    \left\{
    (x_1,\ldots,x_n)\in\mathbb{R}^n:
    \exists\,x_{n+1}\in\mathbb{R},
    (x_1,\ldots,x_n,x_{n+1})\in T
    \right\},
    \]
    belongs to $\mathcal{M}_n$.

    \item[(iv)] For every polynomial function $p\colon \mathbb{R}^n \to \mathbb{R}$, its zero set belongs to $\mathcal{M}_n$; that is,
    \[
    \{x \in \mathbb{R}^n: p(x) = 0\}\in \mathcal{M}_n.
    \]
    \item[(v)] The sets in $\mathcal{M}_1$ are exactly the finite unions of intervals and points.
\end{enumerate}
\end{definition}

A set $A$ is said to be definable in $\mathcal{M}$ if $A\in\mathcal{M}$. This notion naturally extends to functions through their graphs, leading to the following definition of definable functions.
\begin{definition}(Definable functions)
A function $f:\mathbb{R}^n\to\mathbb{R}\cup\{+\infty\}$ is called definable in $\mathcal{M}$ if its graph
\[
\operatorname{gph} f
:=
\{(x,f(x)):x\in\operatorname{dom}f\}
\]
is a definable subset of $\mathbb{R}^{n+1}$, i.e., $\operatorname{gph} f \in \mathcal{M}_{n+1}$.
\end{definition}
As discussed in \cite{attouch2010proximal}, functions definable in an $o$-minimal structure inherit many of the favorable properties enjoyed by semialgebraic functions. In particular, the class of definable functions is closed under several fundamental operations, including finite sums, composition, generalized inversion, and indicator functions of definable sets.

The $o$-minimal framework encompasses a remarkably broad class of functions that commonly arise in optimization and imaging applications, including (see \cite{van1996geometric}):
\begin{itemize}
\item Semialgebraic functions, which can be interpreted as functions definable in the $o$-minimal structure $\mathcal{M}(\mathbb{R},+,\cdot,(r)_{r\in \mathbb{R}})$, namely the smallest structure in $(\mathbb{R},+,\cdot)$ containing all singletons. Indeed, by the Tarski-Seidenberg theorem, the collection of semialgebraic sets coincides with $\mathcal{M}(\mathbb{R},+,\cdot,(r)_{r\in \mathbb{R}})$.
\item Globally subanalytic functions \cite{bolte2007lojasiewicz}, which are definable in the structure $\mathcal{M}(\mathbb{R}_{\text{an}})$, where $\mathbb{R}_{\text{an}}=(\mathbb{R},+,\cdot,(f))$ and $f$ ranges over restricted analytic functions.
\item Log-exp functions, which are definable in the structure $\mathcal{M}(\mathbb{R}_{\text{an,exp}})$, related to the field $\mathbb{R}_{\text{an,exp}}=(\mathbb{R},+,\cdot,(f),\text{exp})$, where $f$ ranges over restricted analytic functions and $\text{exp}(x)=e^x$.
\end{itemize}
These definable structures satisfy the following inclusion relationship:
\begin{equation*}
\mathcal{M}(\mathbb{R},+,\cdot,(r)_{r\in \mathbb{R}})\subseteq \mathcal{M}(\mathbb{R}_{\text{an}}) \subseteq \mathcal{M}(\mathbb{R}_{\text{an,exp}}).
\end{equation*}
For a comprehensive introduction to $o$-minimal structures, we refer the reader to \cite{van1996geometric,bolte2007clarke}. Among their numerous applications in optimization theory, one of the most important is the extension of the Kurdyka-\L{}ojasiewicz framework.
\begin{theorem}\label{theo:ominKL}
Suppose that $f:\mathbb{R}^{n}\rightarrow\mathbb{R}\cup\{+\infty\}$ is a proper lower semicontinuous function definable in an $o$-minimal structure $\mathcal{M}$. Then, for every $x\in\operatorname{dom}\partial f$, the function $f$ satisfies the Kurdyka--{\L}ojasiewicz property at $x$.
\end{theorem}

\section*{Acknowledgments}
The research of Luca Ratti was partially supported by the PNRR project M4C2-PE00000013 (FAIR), Spoke 8 ``Pervasive AI'', funded by the NextGeneration EU programme (CUP J33C22002830006), and by the INDAM-GNCS 2026 project (CUP E53C25002010001). The research of Zhichang Guo was supported by the National Key Research and Development Program of China (Grant No. 2024YFE0216800).

\bibliographystyle{abbrv}

\bibliography{cas-refs}



\end{document}